\documentclass[pdflatex,sn-mathphys-num]{sn-jnl}

\usepackage{graphicx}%
\usepackage{multirow}%
\usepackage{amsmath,amssymb,amsfonts}%
\usepackage{amsthm}%
\usepackage{mathrsfs}%
\usepackage{subcaption}
\usepackage[title]{appendix}%
\usepackage{xcolor}%
\usepackage{textcomp}%
\usepackage{manyfoot}%
\usepackage{booktabs}%
\usepackage{algorithm}%
\usepackage{algorithmicx}%
\usepackage{algpseudocode}%
\newtheorem{lemma}{Lemma}
\usepackage{bm}%
\usepackage{listings}%
\theoremstyle{thmstyleone}%
\newtheorem{theorem}{Theorem}
\theoremstyle{thmstyletwo}%
\newtheorem{remark}{Remark}%

\theoremstyle{thmstylethree}%
\numberwithin{equation}{section}
\numberwithin{equation}{section}
\begin{document}

\title[Article Title]{A Priori Error Analysis of a High-Order Selective Discontinuous Galerkin Method for Elliptic Interface Problems}


\author[1]{\fnm{Fang} \sur{Liu}}\email{fliu@okstate.edu}

\author[2]{\fnm{Haroun} \sur{Meghaichi}}\email{meghaichi.1@osu.edu}

\author*[1]{\fnm{Xu} \sur{Zhang}}\email{xzhang@okstate.edu}

\affil*[1]{\orgdiv{Department of Mathematics}, \orgname{Oklahoma State University}, \orgaddress{\city{Stillwater}, \state{OK} \postcode{74078}, \country{USA}}}

\affil[2]{\orgdiv{Department of Mathematics},  \orgname{Ohio State University}, \orgaddress{\city{Columbus}, \state{OH} \postcode{43210}, \country{USA}}}


\abstract{This paper develops a high-order selective discontinuous Galerkin (SDG) method for solving elliptic interface problems on interface-unfitted Cartesian meshes. This method applies the discontinuous Galerkin (DG) formulation on interface elements and the continuous Galerkin (CG) formulation elsewhere. Correspondingly, we construct a new, locally conforming, hybrid immersed finite element (HIFE) space based on the high-order Frenet IFE basis functions of \cite{2024AdjeridLinMeghaichi}. Compared with the DG method, the computational cost of this SDG method is significantly reduced and remains comparable to that of the CG method. We prove that the new HIFE space achieves optimal approximation under $h$-refinement, and we establish the well-posedness of the SDG scheme. {\it A priori} error estimates are derived in the energy and $L^2$ norms. Numerical examples are provided to verify the theoretical analysis.}

\keywords{Selective discontinuous Galerkin, Frenet immersed finite element, Hybrid immersed finite element space, Error analysis, High-order method.}

\pacs[MSC Classification]{35R05, 65N30, 97N50}

\maketitle

\section{Introduction}
Let $\Omega\subset\mathbb{R}^2$ be a rectangular domain with boundary $\partial \Omega$. Suppose that $\Omega$ is split by an interface curve $\Gamma$ into two subdomains $\Omega^-$ and $\Omega^+$ such that $\Gamma=\overline{\Omega^-}\cap\overline{\Omega^+}$.
Consider the elliptic boundary value problem
 \begin{align}
-\nabla \cdot(\beta\nabla u)=f,&\quad\text{ in } \Omega^-\cup \Omega^+,\label{1.1}\\
u=g, &\quad\text{ on } \partial\Omega.\label{1.2}
\end{align}
Across the interface $\Gamma$, the solution is assumed to satisfy the jump conditions: 
\begin{align}
\left[ u \right]_{\Gamma}=0,~~~
\left[ \beta\frac{\partial u}{\partial \mathbf{n}} \right]_{\Gamma}=0,\label{1.3}
\end{align} 
where the jump $[v]_{\Gamma}:=v^+|_{\Gamma}-v^-|_{\Gamma}$ with $v^{\pm}=v|_{\Omega^{\pm}}$. $\textbf{n}$ denotes the normal vector of the interface $\Gamma$ pointing from $\Omega^-$ to $\Omega^+$, and the diffusion coefficient $\beta$ is a piecewise constant
 \begin{align*}
 \beta=
 \left\{\begin{aligned}
&\beta^-,&\text{ in } \Omega^-,\\
&\beta^+,&\text{ in } \Omega^+.
\end{aligned}\right.
\end{align*}
For high-order approximations (i.e., $m\geq 2$), we further assume $f\in C^{m-2}(\Omega)$ and impose the Laplacian extended jump conditions
\begin{equation}
\left[ \beta\frac{\partial^j{\Delta} u}{\partial\mathbf{n}^j} \right]_\Gamma=0,\quad j=0,1,\cdots,m-2.\label{1.4}
\end{equation}

The finite element method is one of the most widely used frameworks for numerical approximation of partial differential equations. Based on the treatment of global continuity, finite element formulations are broadly categorized as continuous Galerkin (CG) \cite{1973Strang,1978Ciarlet,1994Brenner} and discontinuous Galerkin (DG) \cite{1990CockburnHouShu, 2001ArnoldBrezziBernardoMarini, 2008Riviere,2019LinSheenZhang} methods. Compared with the CG method, the DG method has more global degrees of freedom (DoFs) because of discontinuities across elements. 
On the other hand, the CG method enforces global continuity of basis functions, which introduces additional challenges in both the construction and the error analysis for interface problems.
In 2014, a selective DG (SDG) method was proposed in \cite{2014HeLinLin}, which uses the interior penalty discontinuous Galerkin (IPDG) formulation where local refinement is needed and the standard Galerkin formulation everywhere else. Importantly, the global DoF count of the SDG method is significantly smaller than that of the DG method as the mesh size decreases, while remaining comparable to that of the CG method.

The discontinuity of the diffusion coefficient across the interface reduces the regularity of the solution \cite{1970Ivo, 1998ChenJun}, posing significant challenges for numerical analysis. Standard finite element methods typically require body-fitted meshes, which must be regenerated whenever the interface evolves. To overcome this limitation, immersed finite element (IFE) methods have been developed, which allow the interface to cut through the mesh while incorporating the jump conditions into the basis functions.  In 1998, Li \cite{1998Li} introduced a linear IFE space where the basis functions satisfy the interface jump conditions for the 1D elliptic equation. In \cite{2001LinLinRogersRyan}, the authors proposed a bilinear IFE space whose basis functions approximately satisfy the jump conditions; the approximation capability of this space was later proved in \cite{2008HeLinLin}. A linear IFE space for the 2D model on triangular meshes was proposed in \cite{2004LiLinLinRogers}, where the authors proved its approximation capability using a multi-point Taylor expansion technique. Lin et al. \cite{2015LinLinZhang} proposed a partially penalized immersed finite element (PPIFE) method for elliptic interface problems that overcomes a limitation of the classic IFE method by ensuring optimal convergence rates in the $H^1$- and $L^2$-norm under mesh refinement. 
 
 High-order IFE approximations have also been developed for  elliptic interface problems. In \cite{2006CampLinLinSun}, the authors introduced three types of 1D quadratic IFE spaces and their 2D tensor-product extensions when the interface is linear. Adjerid and Lin \cite{2009SlimaneLin} presented a high-order IFE space for the 1D model that satisfies the extended jump conditions. Then, in \cite{2014AdjeridMohamedLin}, the authors constructed general high-order IFE spaces when the interface is a straight line for the 2D case. In addition, Adjerid et al. \cite{2017AdjeridGuoLin} constructed an IFE space by using a least-squares formulation that enforces all interface jump conditions weakly along the actual interface. Later, Guo and Lin \cite{2019GuoLin} proposed a high-order IFE space based on a Cauchy extension; this approach enforces the  interface jump conditions weakly and guarantees the existence and uniqueness of the IFE shape functions. Adjerid et al. \cite{2024AdjeridLinMeghaichi,2025AdjeridLinMeghaichi, 2025AdjeridLinMeghaichi2} proposed the high-order geometry-conforming (GC) IFE spaces based on the Frenet-Serret transformation on rectangular meshes. These new IFE functions satisfy the interface conditions exactly and preserve optimal approximation capabilities. Recently, this GC-IFE construction has been extended to triangular meshes in \cite{2025LinLinXu}.

 The analysis of IFE methods is particularly challenging, because of the low regularity of solutions to interface problems and the constrained regularity inherited by the local approximation spaces.
In 2014, He et al. \cite{2012HeLinLin} analyzed the error of both the bilinear and linear IFE solutions for 2D elliptic interface problems. An {\it a priori} error analysis for the PPIFE method was developed in \cite{2015LinLinZhang}, and an improved error analysis with weaker global regularity requirements was provided in \cite{2019GuoLinZhuang}. Moreover, in 2015,  Lin et al. \cite{2015LinYangZhang} derived {\it a priori} error estimates for a class of interior penalty DG methods using IFE functions. 
An {\it a priori} analysis for a high-degree IFE method based on a Cauchy extension was given in \cite{2019GuoLin}. In 2019, Lin et al. \cite{2019LinSheenZhang} derived the error estimates of a nonconforming IFE method.  Error estimates for an enriched IFE method for interface problems
with nonhomogeneous jump conditions were presented in \cite{2023AdjeridIvoGuoLin, 2025GuoZhang}. Recently, Meghaichi et al. \cite{2025AdjeridLinMeghaichi} presented an error analysis of the Frenet IFE method for elliptic interface problems.
 
In this work, we develop a high-order SDG method for elliptic interface problems in two dimensions and derive {\it a priori} error estimates for this method. The scheme is built upon hybrid IFE (HIFE) spaces constructed as follows: standard Lagrange nodal basis functions are used on non-interface elements, while Frenet IFE bases are adopted on interface elements. This hybrid IFE structure has the following advantages:
\begin{itemize}
\item Compared with standard DG methods, the proposed scheme yields significant savings in computational costs. Since the CG formulation is used on non-interface elements, which constitute the majority of all elements, the cost reduction becomes increasingly pronounced as the mesh is refined. A detailed comparison is reported in Section \ref{sec5}. 
\item The new HIFE space is locally $H^1$-conforming because the Frenet IFE basis functions satisfy the jump conditions precisely. This is in contrast to most existing approaches, in which the jump conditions are satisfied only approximately or in a weak sense.
\item We prove that the HIFE spaces possess optimal approximation capability and, via a rigorous {\it a priori} error analysis, show that the corresponding SDG methods achieve optimal convergence rates in both the energy norm and the standard $L^2$ norm.
\end{itemize}

Although our construction relies on the Frenet IFE spaces introduced in \cite{2025AdjeridLinMeghaichi}, the present work differs in three essential ways. First, \cite{2025AdjeridLinMeghaichi} analyzes a fully DG-IFE method on all elements, whereas we develop a hybrid scheme in which the DG formulation is restricted to interface elements and the CG formulation is used elsewhere. Second, the penalty in our bilinear form is imposed on a much smaller edge set than the full set of interior edges, which requires a separate well-posedness argument. Third, the hybrid IFE space is smaller, and in fact, a subspace of the DG-IFE space in \cite{2025AdjeridLinMeghaichi}; its approximation properties and the corresponding {\it a priori} estimates therefore do not follow from \cite{2025AdjeridLinMeghaichi} and are established here. 

The remainder of this article is organized as follows. In Section \ref{sec2}, we define some notation and review the local Frenet IFE spaces. In Section \ref{sec3}, we derive the high-order SDG method and prove the well-posedness of the scheme. In Section \ref{sec4}, we show that the SDG schemes achieve optimal convergence rates in the energy and $L^2$ norms. In Section \ref{sec5}, we  present numerical examples to verify the theoretical analysis. A brief conclusion is given in Section \ref{sec6}.

\section{Preliminaries}\label{sec2}
In this section, we introduce some necessary notation and recall the local Frenet IFE space introduced in \cite{2024AdjeridLinMeghaichi,2025AdjeridLinMeghaichi}.
\subsection{Notation}\label{sub2.1}
  Let $\hat{\Omega}$ be any subset of $\Omega$, and let $W^{k}_p(\hat{\Omega})$ denote the standard Sobolev space with the norm $\lVert\cdot\rVert_{W^{k}_p(\hat{\Omega})}$ and the semi-norm $|\cdot|_{W^{k}_p(\hat{\Omega})}$, where $k\geq 0$ is an integer and $1\leq p\leq\infty$. When $p=2$, we use $H^{k}(\hat{\Omega})$ to denote $W^{k}_2(\hat{\Omega})$ with the norm $\lVert\cdot\rVert_{H^{k}(\hat{\Omega})}$ and the semi-norm $|\cdot|_{H^{k}(\hat{\Omega})}$. For convenience, $(\cdot,\cdot)_{\partial\hat{\Omega}}$ and $\lVert\cdot\rVert_{L^2(\partial\hat{\Omega})}$ denote the inner product and the norm on $L^2(\partial\hat{\Omega})$, respectively.
If $\hat{\Omega}$ intersects $\Gamma$, we denote the subdomains by $\hat{\Omega}^{\pm}=\hat{\Omega}\cap\Omega^{\pm}$. 
For $m>\frac{3}{2}$, we define the following Sobolev space 
\begin{equation}\label{eq: Hm}
\mathcal{H}^{m}(\hat{\Omega},\Gamma;\beta)=\Big\{v\in L^2(\hat{\Omega}):v|_{\hat{\Omega}^{\pm}}\in H^{m}(\hat{\Omega^\pm}),[v]_{\Gamma\cap\hat{\Omega}}=[\beta\partial_{\textbf{n}}v]_{\Gamma\cap\hat{\Omega}}=0\Big\},
\end{equation}
equipped with the norm $\|\cdot\|^2_{\mathcal{H}^m(\hat{\Omega})}=\|\cdot\|^2_{\mathcal{H}^m(\hat{\Omega}^+)}+\|\cdot\|^2_{\mathcal{H}^m(\hat{\Omega}^-)}$ and the semi-norm
$|\cdot|^2_{\mathcal{H}^m(\hat{\Omega})}=|\cdot|^2_{\mathcal{H}^m(\hat{\Omega}^+)}+|\cdot|^2_{\mathcal{H}^m(\hat{\Omega}^-)}$.

Let $\mathcal{T}_h$ be a Cartesian rectangular mesh of $\Omega$ with mesh size $h = \max_{K \in \mathcal{T}_h} \operatorname{diam}(K)$, and let $\mathcal{E}_h$ denote the set of edges of $\mathcal{T}_h$. We assume $h$ is sufficiently small so that the following hypotheses hold:
\begin{itemize}
\item $\mathbf{H_1}$: The interface $\Gamma$ intersects each edge in at most one point, unless that edge is part of $\Gamma$.
\item $\mathbf{H_2}$: The interface $\Gamma$ intersects each element in at most two edges.  
\end{itemize}


We partition the mesh as $\mathcal{T}_h = \mathcal{T}_h^{i} \cup \mathcal{T}_h^{n}$, 
where $\mathcal{T}_h^{i}$ and $\mathcal{T}_h^{n}$ denote the sets of interface and non-interface elements, respectively. 
Let $\mathring{\mathcal{E}}_h$ be the set of interior edges of $\mathcal{T}_h$. We denote by $\mathring{\mathcal{E}}_h^{i}$ and $\mathring{\mathcal{E}}_h^{n}$ the sets of interior interface edges and interior non-interface edges, respectively. In addition,  $\mathring{\mathcal{E}}_h^s$ denotes the collection of edges of all interface elements, and $\mathcal{T}_h^s$ the set of elements that have at least one edge contained in $\mathring{\mathcal{E}}_h^s$. Finally, we let 
$\mathcal{N}_h$ be the set of all vertices and let $\mathcal{N}_h^i$  be the set of vertices of all interface elements. 

For every interior edge $B\in\mathring{\mathcal{E}}_h$, let $K_{B,1}$
 and $K_{B,2}$ denote the two elements sharing $B$. For a function $u$ defined on $K_{B,1}\cup K_{B,2}$, we define the average and jump of $u$ on $B$ by
\[
\{u\}_B = \frac{1}{2} \Big( (u|_{K_{B,1}})|_B + (u|_{K_{B,2}})|_B \Big), 
\qquad
[u]_B = (u|_{K_{B,1}})|_B - (u|_{K_{B,2}})|_B.
\]
For $m>\frac{3}{2}$, we define the broken Sobolev space on the mesh $\mathcal{T}_h$
 \begin{equation}
\begin{aligned}
\mathcal{H}^m(\mathcal{T}_h,\Gamma;\beta)=\Big\{v:&~v|_{K}\in H^m(K),~\forall K \in \mathcal{T}_h^n, \text{ and } v|_{K} \in\mathcal{H}^{m}(K,\Gamma;\beta)~\forall K \in \mathcal{T}_h^i;\\
&~v \text{ is continuous across each edge } B\in\mathring{\mathcal{E}}_h\setminus \mathring{\mathcal{E}}_h^s;~v|_{\partial\Omega}=0\Big\}.
\label{2.1}
\end{aligned}
\end{equation}
\begin{remark}
Following the definition of \eqref{eq: Hm} and \eqref{2.1}, we have $\mathcal{H}^m(\mathcal{T}_h,\Gamma;\beta)\subset \mathcal{H}^m(\Omega,\Gamma;\beta)$, and for each function $v\in\mathcal{H}^m(\mathcal{T}_h,\Gamma;\beta)$, it is continuous at nodes $X\in \mathcal{N}_h\setminus\mathcal{N}_h^i$.
\end{remark}

\subsection{Geometry Conforming IFE Spaces via Frenet Mapping}\label{sub2.2}
In this subsection, we briefly review the Frenet transformation and the local Frenet IFE space introduced in \cite{2024AdjeridLinMeghaichi, 2025AdjeridLinMeghaichi}. 
Assume that the interface $\Gamma$ is a simple connected $C^2$ curve, which can be parametrized by  $\mathbf{g}:~[\xi_s,\xi_e]\rightarrow \Omega$ such that

\begin{equation*}
\mathbf{g}(\xi)=(g_1(\xi),g_2(\xi))^T.
\end{equation*}
The unit tangent vector $\bm{\tau}(\xi)$, the unit normal vector $\mathbf{n}(\xi)$, and the signed curvature $\kappa(\xi)$ at the point $\mathbf{g}(\xi)\in\Gamma$ are defined by
\begin{equation*}
\bm{\tau}(\xi)=\frac{1}{\|\mathbf{g}'(\xi)\|}\mathbf{g}'(\xi),\quad
\mathbf{n}(\xi)=\begin{bmatrix}
0 & 1 \\
-1 & 0
\end{bmatrix}
\bm{\tau}(\xi),\quad
\mathbf{\kappa}(\xi)=\frac{1}{\|\mathbf{g}'(\xi)\|^3}(\mathbf{g}'(\xi))^{T}\begin{bmatrix}
0 & 1 \\
-1 & 0
\end{bmatrix}\mathbf{g}''(\xi).
\end{equation*}
A curve parallel to the interface $\Gamma$ with an offset distance $\eta$ from $\Gamma$ is represented as follows
\begin{equation*}
\mathbf{x}(\eta,\xi)=\begin{bmatrix}
x(\eta,\xi)  \\
y(\eta,\xi) 
\end{bmatrix}=P_{\Gamma}(\eta,\xi) =\mathbf{g}(\xi)+\eta\mathbf{n}(\xi),~\forall \xi\in[\xi_s,\xi_e].
\end{equation*}
There exists a tubular neighborhood $N_{\Gamma}(\epsilon)$ with some $\epsilon>0$ denoted by
\begin{equation*}
N_{\Gamma}(\epsilon)=P_{\Gamma}([-\epsilon,\epsilon]\times[\xi_s,\xi_e]),
\end{equation*} 
such that $P_\Gamma$ is a bijection on $N_{\Gamma}(\epsilon)$. Hence, there exists an inverse transformation $R_{\Gamma}:~N_{\Gamma}(\epsilon)\rightarrow[-\epsilon,\epsilon]\times[\xi_s,\xi_e]$ such that 
\begin{equation*}
\begin{bmatrix}
\eta  \\
\xi 
\end{bmatrix}=
\begin{bmatrix}
\eta(x,y)  \\
\xi(x,y) 
\end{bmatrix}=
R_{\Gamma}(x,y)=P^{-1}_{\Gamma}(x,y)\in[-\epsilon,\epsilon]\times[\xi_s,\xi_e],~\forall\mathbf{x}\in N_{\Gamma}(\epsilon).
\end{equation*}

Assume that the mesh size $h$ is sufficiently small such that $N_{\Gamma}(h)\subset N_{\Gamma}(\epsilon)$. Then all the interface elements are within the tubular neighborhood $N_{\Gamma}(h)$ of $\Gamma$. Let $K\in\mathcal{T}_h^i$ be an interface element with vertices 
$A_i\in N_{\Gamma}(\epsilon),i=1,2,3,4$. Then there exists a unique $\xi_i\in[\xi_s,\xi_e]$ such that $\|A_i-\mathbf{g}(\xi_i)\|=\operatorname{dist}(\Gamma,A_i).$  We form the rectangle $\hat{K}_F=[-h,h]\times[a_K,b_K]$ and define the fictitious element $K_F$ of the interface element $K$ as 
\begin{equation*}
K_F=P_{\Gamma}(\hat{K}_F)=P_{\Gamma}
\left([-h,h]\times[a_K,b_K]\right),
\end{equation*}
where $a_K=\min(\xi_1,\xi_2,\xi_3,\xi_4),~b_K=\max(\xi_1,\xi_2,\xi_3,\xi_4)$. 
Figure \ref{fig1} is an illustration of the Frenet transformation $P_\Gamma$ and its inverse $R_\Gamma$.  

\begin{figure}[!htb]
    \centering
    \includegraphics[width=0.75\textwidth]{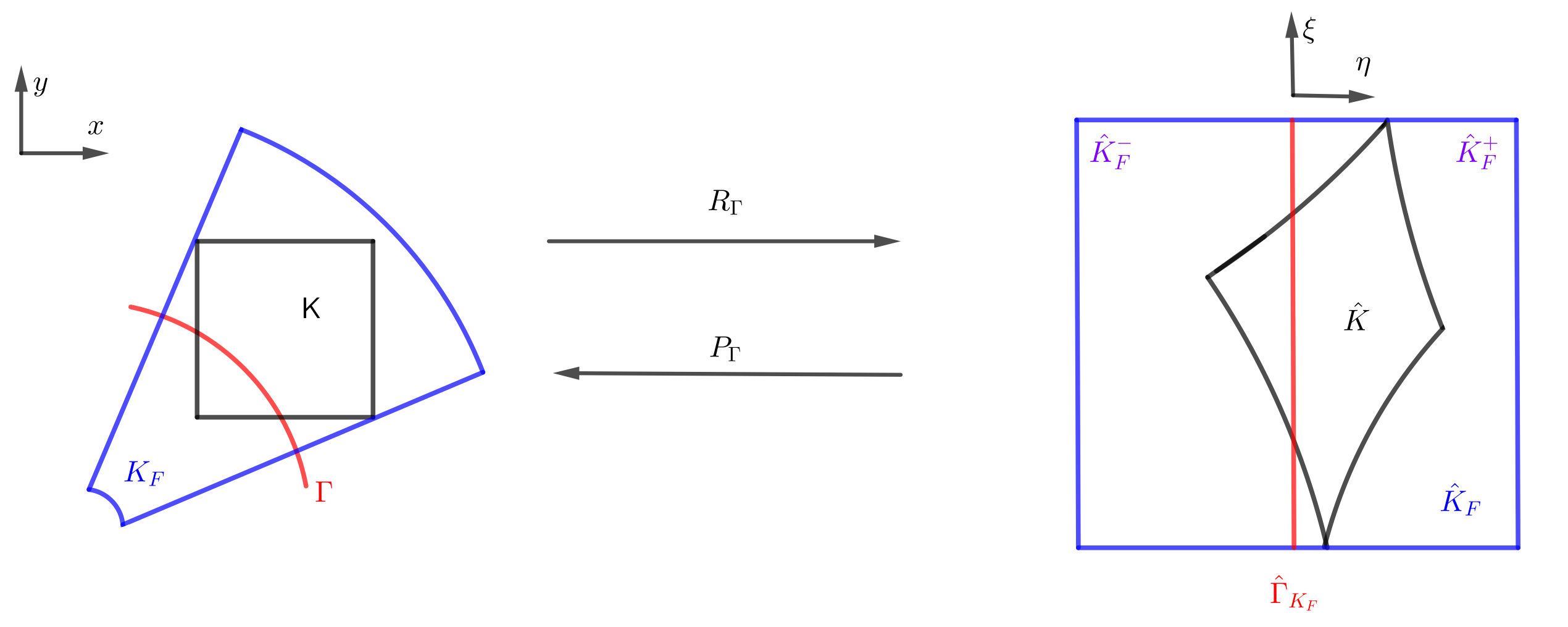}
    \caption{An illustration of the Frenet transformation $P_{\Gamma}$ and its inverse $R_{\Gamma}$ on an interface element $K$.}
    \label{fig1}
\end{figure}

Denote $\hat{u}=u\circ P_{\Gamma}$ and $\hat{\beta}=\beta\circ P_{\Gamma}$, then $\hat{u}(\eta,\xi)$ satisfies the following interface jump conditions:
\begin{equation}\label{2.3}
[ \hat{u} ]_{\hat{\Gamma}_{K_F}} = 0,~~
[ \hat{\beta}\,\hat{u}_{\eta} ]_{\hat{\Gamma}_{K_F}} = 0,
\end{equation}
and
\begin{equation}\label{2.3.2}
\left[ \hat{\beta}\,\frac{\partial^j}{\partial \eta^j}\mathcal{L}(\hat{u}) \right]_{\hat{\Gamma}_{K_F}} = 0,
\qquad j = 0,1,\dots,m-2 .
\end{equation}
Here, $\mathcal{L}(\hat{u}(\eta,\xi)) = \Delta u(\mathbf{x})$ denotes the Laplacian in $(\eta,\xi)$ coordinates that takes the form
\begin{equation}
\begin{aligned}\label{2.2}
\mathcal{L}(\hat{u}(\eta,\xi)):=\hat{u}_{\eta\eta}(\eta,\xi)+J_0(\eta,\xi)\hat{u}_{\xi\xi}(\eta,\xi)+J_1(\eta,\xi)\hat{u}_{\eta}(\eta,\xi)+J_2(\eta,\xi)\hat{u}_{\xi}(\eta,\xi),
\end{aligned}
\end{equation}
where 
\begin{equation*}
\begin{aligned}
&J_0(\eta,\xi)=\left(\frac{\psi(\eta,\xi)}{\lVert\mathbf{g}'(\xi)\rVert}\right)^2,~~J_1(\eta,\xi)=\kappa(\xi)\psi(\eta,\xi),\\&J_2(\eta,\xi)=-\left(\frac{\psi(\eta,\xi)}{\lVert\mathbf{g}'(\xi)\rVert}\right)^2\left(\eta\kappa'(\xi)\psi(\eta,\xi)+\frac{\mathbf{g}'(\xi)\cdot\mathbf{g}''(\xi)}{\lVert\mathbf{g}'(\xi)\rVert^2}\right),\\
&\psi(\eta,\xi)=(1+\eta\kappa(\xi))^{-1}.
\end{aligned}
\end{equation*}

For a given set $I\subset\mathbb{R}$, let $P_m(I)$ be the space of polynomials of degree up to $m$ on $I$. Imposing the jump conditions \eqref{2.3}-\eqref{2.3.2} weakly yields
\begin{align}
\int_{\hat{\Gamma}_{K_F}}
[ \hat{\phi} ] v\,ds &= 0, \quad
\int_{\hat{\Gamma}_{K_F}}
[ \hat{\beta}\,\hat{\phi}_{\eta} ] v\,ds = 0,
\quad \forall ~v \in P_m(\hat{\Gamma}_{K_F}),\label{2.4} \\
\int_{\hat{\Gamma}_{K_F}}
\left[\hat{\beta}
\frac{\partial^{j}}{\partial \eta^{j}}
\mathcal{L}(\hat{\phi})
\right]
v\,ds &= 0,
\quad \forall v \in P_m(\hat{\Gamma}_{K_F}),
\quad j = 0,1,\dots,m-2 , \label{2.5}
\end{align}
The conditions in \eqref{2.4} are equivalent to those in \eqref{2.3}, while the condition in \eqref{2.5} is equivalent in a weak sense to \eqref{2.3.2}. 
The following local Frenet IFE space is well defined on the fictitious element $K_F$:
\begin{equation*}
\hat{\mathcal{V}}^{m}_{\beta}(\hat{K}_{F})
=
\left\{
\hat{\phi} : \hat{K}_{F} \to \mathbb{R}
\;\middle|\;
\hat{\phi}|_{\hat{K}_{F}^{\pm}} \in Q_{m}(\hat{K}_{F}^{\pm})
\ \text{and}\ 
\hat{\phi} \ \text{satisfies \eqref{2.4}-\eqref{2.5}}
\right\}.
\end{equation*} 
Using the inverse Frenet map $R_{\Gamma}$ on the fictitious elements $K_F$, we get the corresponding local IFE space defined on the interface element $K$ as follows:
\begin{equation*}
\mathcal{V}^{m}_{\beta}(K)
=
\left\{
\hat{\phi}\circ R_{\Gamma}\big|_{K}
\;\middle|\;
\hat{\phi}\in
\hat{\mathcal{V}}^{m}_{\beta}(\hat{K}_{F})
\right\}.
\end{equation*}

\begin{remark}
    The local IFE space $\mathcal{V}^{m}_{\beta}(K)$ has a dimension of $(m+1)^2$, same as the standard tensor-product space $Q_m$, although the functions in $\mathcal{V}^{m}_{\beta}(K)$ are no longer piecewise polynomials. For more details on constructing Frenet-IFE basis functions, we refer to \cite{2024AdjeridLinMeghaichi}.
\end{remark}


\section{High-Order Selective Discontinuous Galerkin Method}\label{sec3}
In this section, we introduce a new high-order hybrid IFE space, derive the selective discontinuous Galerkin method, and prove the well-posedness of the resulting SDG scheme. Throughout the paper, we assume without loss of generality that $\beta^-\leq \beta^+$.
\subsection{SDG Method}
On each element $K\in\mathcal{T}_h$, we define the local IFE space
\begin{equation*}
S_h^m(K) =
\begin{cases}
Q_m(K), & \text{if } K\in\mathcal{T}_h^n,\\
\mathcal{V}^{m}_{\beta}(K), & \text{if } K\in\mathcal{T}_h^i.\\
\end{cases}
\end{equation*}
The high-order HIFE space on  $\Omega$ is defined as follows
\begin{equation*}
S_h^m(\Omega)=   \{v\in\mathcal{H}^m(\mathcal{T}_h,\Gamma;\beta):v|_{K}\in S ^  m _h(K),~\forall K \in \mathcal{T}_h\}.
\end{equation*}

Clearly, $S_h^m(\Omega)$ is a subspace of $\mathcal{H}^m(\mathcal{T}_h,\Gamma;\beta)$. We now describe the SDG method for the interface problem \eqref{1.1}-\eqref{1.4}: find $u_h \in S_h^m(\Omega)$ such that
\begin{equation}
a_h(u_h,v_h)=(f,v_h),\quad\forall v_h\in S_h^m(\Omega), \label{3.1}
\end{equation}
where the bilinear form $a_h:S_h^m(\Omega)\times S_h^m(\Omega)\rightarrow \mathbb{R}$ is defined by
\begin{equation}
\begin{aligned}
a_h(u,v)=
&\sum_{K\in\mathcal{T}_h}\int_{K}\beta\nabla u\cdot \nabla v\,dX-\sum_{B\in\mathring{\mathcal{E}}_h^s}\int_B\{\beta\nabla u\cdot\mathbf{n}_B\}[v]\,ds\\
&+\epsilon\sum_{B\in\mathring{\mathcal{E}}_h^s}\int_B\{\beta\nabla v\cdot\mathbf{n}_B\}[u]\,ds+\sum_{B\in\mathring{\mathcal{E}}_h^s}\int_B\frac{\sigma_B^0}{|B|^\alpha}[u][v]ds.\label{3.2}
\end{aligned}
\end{equation}

Compared with the DG method \cite{2024AdjeridLinMeghaichi}, the SDG method applies the Frenet IFE basis functions only on the interface elements, rather than throughout the entire domain.  Unlike the CG method, the SDG method applies the nodal basis functions only on the non-interface elements. Figure~\ref{fig2} shows the distribution of global DoFs for the DG, SDG, and CG methods. The detailed quantitative comparison is reported in Section \ref{sec5}. By the  ``CG method" for interface problem we mean the partially penalized IFE scheme \cite{2015LinLinZhang}, whose global DoF structure is isomorphic to the standard CG method. 
Compared with the CG and DG methods, the penalty in our SDG method is imposed on all edges of interface elements (i.e., $\mathring{\mathcal{E}}_h^s$), rather than on all interface edges $\mathring{\mathcal{E}}_h^i$ as in the CG method, or on all interior edges $\mathring{\mathcal{E}}_h$ as in the DG method.
These edges are shown in red in Figure~\ref{fig2}. Red edges mark the edges on which the penalty terms are imposed for each method. The blue and green dots represent the DoFs of the Frenet IFE basis and the FE nodal basis, respectively.
\begin{figure}[H]
    \centering
    
    \begin{subfigure}{0.3\textwidth}
        \centering
        \includegraphics[width=\linewidth]{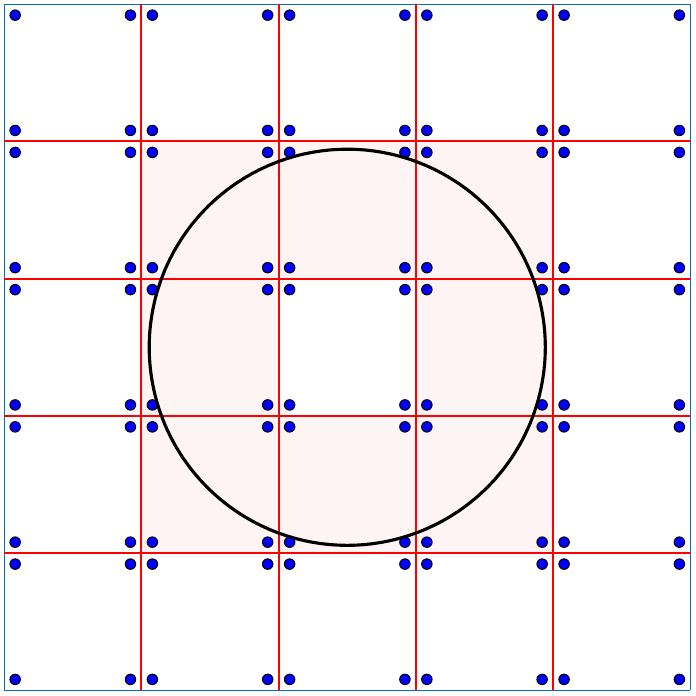}
        \caption{DG Method}
    \end{subfigure}
    \hfill
    \begin{subfigure}{0.3\textwidth}
        \centering
        \includegraphics[width=\linewidth]{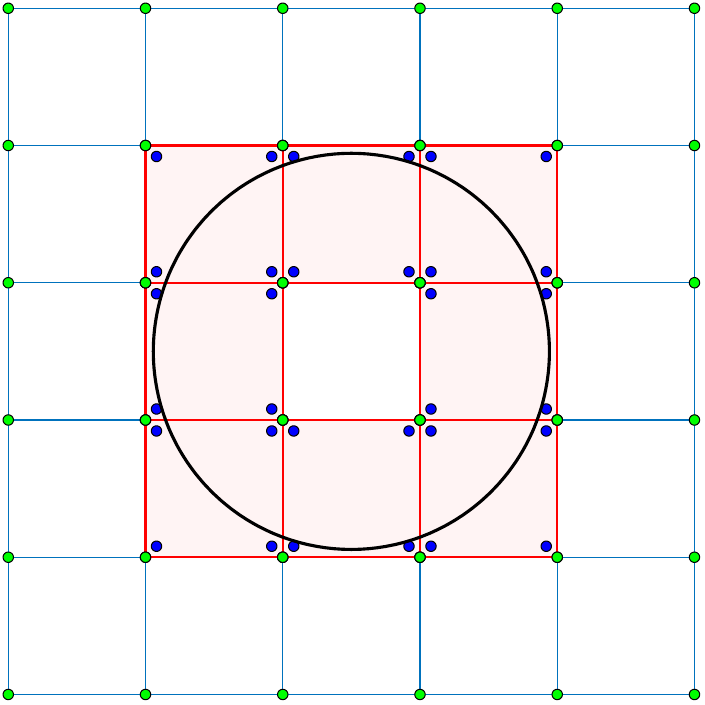}
        \caption{SDG Method}
    \end{subfigure}
    \hfill
    \begin{subfigure}{0.3\textwidth}
        \centering
        \includegraphics[width=\linewidth]{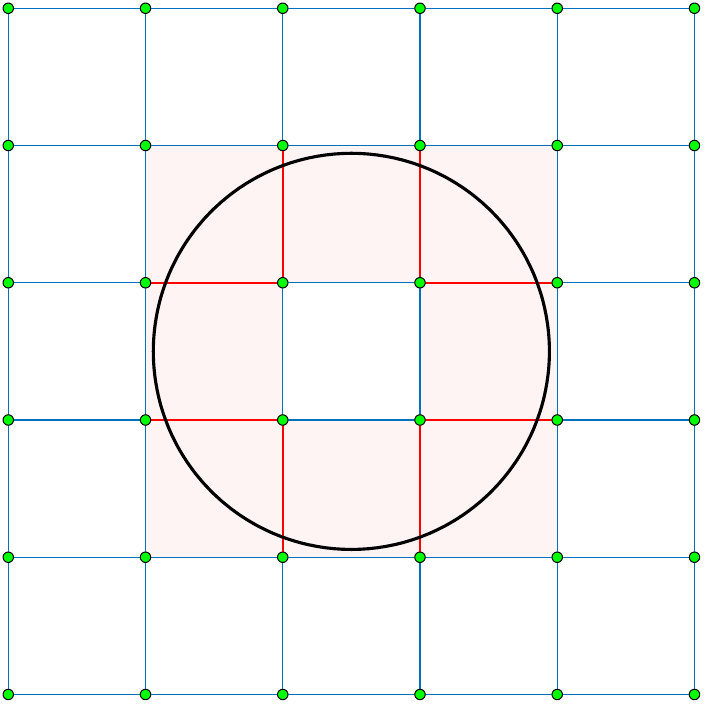}
        \caption{CG (PPIFE) Method}
    \end{subfigure}
    
    \caption{Distribution of global and local DoFs for the DG, SDG, and CG (PPIFE) methods.  }
    \label{fig2}
\end{figure}

\subsection{Well-posedness of SDG method}
We now prove the well-posedness of the SDG method \eqref{3.1} by establishing the coercivity and continuity of the bilinear form $a_h(\cdot,\cdot)$ defined in \eqref{3.2} on $S^m_h(\Omega)$ with respect to the following energy norm
\begin{equation}
\begin{aligned}
\|v\|^2_{h}:=&\sum_{K\in\mathcal{T}_h}\int_{K}\beta\nabla v\cdot \nabla v\,dX+\sum_{B\in\mathring{\mathcal{E}}_h^s}\int_B\frac{\sigma_B^0}{|B|^\alpha}[v][v]ds,\quad\forall v\in S_h^m(\Omega).\label{3.3}
\end{aligned}
\end{equation}

The following lemma shows that $\|\cdot\|_h$, defined in \eqref{3.3}, is indeed a norm on $S^m_h(\Omega)$.
\begin{lemma}\label{lemma1}
Assume that $\sigma_B^0>0$ for all $B\in\mathring{\mathcal{E}}_h^s$ and $\alpha\ge0$, then the functional $\|\cdot\|_{h}$ in \eqref{3.3} defines a norm on $S_h^m(\Omega)$.
\end{lemma}
\begin{proof}
Suppose $\|v\|_{h}=0$, then $\|v\|_{h}^2=0$, i.e.,
\begin{equation*}
\sum_{K\in\mathcal{T}_h}\int_{K}\beta\nabla v\cdot \nabla v\,dX+\sum_{B\in\mathring{\mathcal{E}}_h^s}\int_B\frac{\sigma_B^0}{|B|^\alpha}[v][v]ds=0.
\end{equation*}
Since each term in the above equation is nonnegative, we have
\begin{equation}
\int_{K}\beta\nabla v\cdot \nabla v\,dX=0,~~\forall K\in\mathcal{T}_h,\label{3.4}
\end{equation}
and
\begin{equation}
\int_B\frac{\sigma_B^0}{|B|^\alpha}[v][v]ds=0,~~\forall B\in\mathring{\mathcal{E}}_h^s.\label{3.5}
\end{equation}
From equation \eqref{3.4}, we have
\begin{equation*}
\int_{K}\beta\nabla v\cdot \nabla v\,dX=\int_{K}\beta|\nabla v|^2\,dX=0,~~\forall K\in\mathcal{T}_h.
\end{equation*}
Thus $\nabla v=0,$ on $K\in\mathcal{T}_h$, because $\beta>0$.\\
From equation \eqref{3.5}, we have
\begin{equation*}
\int_B\frac{\sigma_B^0}{|B|^\alpha}[v][v]ds=\int_B\frac{\sigma_B^0}{|B|^\alpha}|[v]|^2ds=0,~~\forall B\in\mathring{\mathcal{E}}_h^s.
\end{equation*}
Thus $[v]=0, $ on $B\in\mathring{\mathcal{E}}_h^s$, since $\sigma_B^0>0$, and $\alpha\geq 0$. 
Therefore, $v$ is constant, with value $C_B$, on the edge $B$ from both sides $K_{B,1}$ and $K_{B,2}$.  Moreover, $v$ is continuous across each $B\in\mathring{\mathcal{E}}_h\setminus\mathring{\mathcal{E}}_h^s$, and $v|_{\partial\Omega}=0$, because $v\in S_h^m(\Omega)$. Hence $v\equiv0$ on $\Omega$.
Since $\|\cdot\|_h$ is easily seen to be a semi-norm, it is a norm.
\end{proof}

The following trace inequality on $S_h^m(K)$ is a direct consequence of Theorem 2 in \cite{2025AdjeridLinMeghaichi}. 
\begin{lemma}\label{lemma2}
For any element  $K\in\mathcal{T}_h$, we have
\begin{equation}
\|\beta\nabla v\|_{L^2(\partial K)}\leq C\frac{\beta^+}{\sqrt{\beta^-}}h^{-\frac{1}{2}}\|\sqrt{\beta}\nabla v\|_{L^2(K)},\quad\forall v\in S_h^m(K).\label{3.6}
\end{equation}
\end{lemma}

\begin{proof}
If $K$ is an interface element, from Theorem 2 in \cite{2025AdjeridLinMeghaichi}, we have
\begin{equation}
\|\beta\nabla v\|_{L^2(\partial K)}\leq C\frac{\beta^+}{\sqrt{\beta^-}}h^{-\frac{1}{2}}\|\sqrt{\beta}\nabla v\|_{L^2(K)},\quad\forall  v\in S_h^m(K).\label{3.7}
\end{equation}
If $K$ is a non-interface element, from \cite{2024AdjeridLinMeghaichi}, we have
\begin{equation}
\|\nabla v\|_{L^2(\partial K)}\leq Ch^{-\frac{1}{2}}\|\nabla v\|_{L^2(K)},\quad\forall v\in S_h^m(K).\label{3.8}
\end{equation}
The conclusion follows directly from  \eqref{3.7}  and \eqref{3.8}.
\end{proof}

The next lemma establishes coercivity of the bilinear form  $a_h(\cdot,\cdot)$ with respect to the energy norm $\|\cdot\|_h$.

\begin{lemma}\label{lemma3}
There exists a constant $\kappa>0$ such that
\begin{equation}
\kappa\|v_h\|_{h}^2\le a_h(v_h,v_h),\quad\forall v_h\in S_h^m(\Omega).\label{3.9}
\end{equation}
holds unconditionally for $\epsilon=1$, and holds for $\epsilon=0$ or $\epsilon=-1$ provided that the stabilization parameter $\sigma_B^0$ in $a_h(\cdot,\cdot)$ is sufficiently large.
\end{lemma}

\begin{proof}
For $\epsilon=1$, coercivity follows directly from the definitions of $a_h(\cdot,\cdot)$ and $\|\cdot\|_h$ with $\kappa=1$.

For $\epsilon=-1,0$, we know
\begin{equation}
\begin{aligned}
a_h(v_h,v_h)=&\sum_{K\in\mathcal{T}_h}\int_{K}\beta\nabla v_h\cdot \nabla v_h\,dX+(\epsilon-1)\sum_{B\in\mathring{\mathcal{E}}_h^s}\int_B\{\beta\nabla v_h\cdot\mathbf{n}_B\}[v_h]\,ds\\&+\sum_{B\in\mathring{\mathcal{E}}_h^s}\int_B\frac{\sigma_B^0}{|B|^\alpha}[v_h][v_h]ds.\label{3.10}
\end{aligned}
\end{equation}
For the second term on the right hand side of \eqref{3.10}, the Cauchy-Schwarz inequality gives
\begin{equation*}
\begin{aligned}
\sum_{B\in\mathring{\mathcal{E}}_h^s}\int_B\{\beta\nabla v_h\cdot\mathbf{n}_B\}[v_h]\,ds
\leq&\sum_{B\in\mathring{\mathcal{E}}_h^s}\|\{\beta\nabla v_h\cdot\mathbf{n}_B\}\|_{L^2(B)}\|[v_h]\|_{L^2(B)}\\
\leq&\sum_{B\in\mathring{\mathcal{E}}_h^s}\|\{\beta\nabla v_h\cdot\mathbf{n}_B\}\|_{L^2(B)}\left(\frac{1}{|B|^{\alpha}}\right)^{\frac{1}{2}-\frac{1}{2}}\|[v_h]\|_{L^2(B)}.
\end{aligned}
\end{equation*}
Consider an interior edge $B\in\mathring{\mathcal{E}}_h^s$ shared by the elements $K_{B,1}$ and $K_{B,2}$. By the triangle inequality and the trace inequality \eqref{3.6}, we obtain
\begin{equation*}
\begin{aligned}
\|\{\beta\nabla v_h\cdot\mathbf{n}_B\}\|_{L^2(B)}
\leq&\frac{1}{2}\|(\beta\nabla v_h\cdot\mathbf{n}_B)|_{K_{B,1}}\|_{L^2(B)}+\frac{1}{2}\|(\beta\nabla v_h\cdot\mathbf{n}_B)|_{K_{B,2}}\|_{L^2(B)}\\
\leq& C\frac{\beta^+}{\sqrt{\beta^-}}h^{-\frac{1}{2}}\Big(\|\sqrt{\beta}\nabla v_h \|_{L^2(K_{B,1})}+\|\sqrt{\beta}\nabla v_h \|_{L^2(K_{B,2})}\Big).
\end{aligned}
\end{equation*}
Then for $0< h\leq 1$, $\alpha\geq 1$, and using $c_1{h}\leq|B|\leq c_2 h$, we have
\begin{equation*}
\begin{aligned}
&\int_B\{\beta\nabla v_h\cdot\mathbf{n}_B\}[v_h]\,ds\\
\leq&C\frac{\beta^+}{\sqrt{\beta^-}}h^{\frac{\alpha}{2}-\frac{1}{2}}\left(\|\sqrt{\beta}\nabla v_h\|^2_{L^2(K_{B,1})}+\|\sqrt{\beta}\nabla v_h\|^2_{L^2(K_{B,2})}\right)^{\frac{1}{2}}\left(\frac{1}{h^{\alpha}}\right)^{\frac{1}{2}}\|[v_h]\|_{L^2(B)}\\
\leq& C\frac{\beta^+}{\sqrt{\beta^-}}\left(\|\sqrt{\beta}\nabla v_h\|^2_{L^2(K_{B,1})}+\|\sqrt{\beta}\nabla v_h\|^2_{L^2(K_{B,2})}\right)^{\frac{1}{2}}\left(\frac{1}{h^{\alpha}}\right)^{\frac{1}{2}}\|[v_h]\|_{L^2(B)}.
\end{aligned}
\end{equation*}
A quadrilateral element has at most $4$ neighbors. Thus, we have
\begin{equation*}
\begin{aligned}
&\sum_{B\in\mathring{\mathcal{E}}_h^s}\int_B\{\beta\nabla v_h\cdot\mathbf{n}_B\}[v_h]\,ds\\
\leq &C\frac{\beta^+}{\sqrt{\beta^-}}
\bigg(\sum_{B\in\mathring{\mathcal{E}}_h^s}\frac{1}{h^{\alpha}}\|[v_h]\|^2_{L^2(B)}\bigg)^{\frac{1}{2}}       
\bigg(\sum_{B\in\mathring{\mathcal{E}}_h^s}\Big(\|\sqrt{\beta}\nabla v_h\|^2_{L^2(K_{B,1})}+\|\sqrt{\beta}\nabla v_h\|^2_{L^2(K_{B,2})}\Big)\bigg)^{\frac{1}{2}}\\
\leq& C\frac{\beta^+}{\sqrt{\beta^-}}\bigg(\sum_{B\in\mathring{\mathcal{E}}_h^s}\frac{1}{h^{\alpha}}\|[v_h]\|^2_{L^2(B)}\bigg)^{\frac{1}{2}}       \bigg(\sum_{K\in\mathcal{T}_h^s}\|\sqrt{\beta}\nabla v_h\|^2_{L^2(K)}\bigg)^{\frac{1}{2}}\\
\leq&\frac{\delta}{2}\sum_{K\in\mathcal{T}_h^s}\|\beta^{\frac{1}{2}}\nabla v_h\|^2_{L^2(K)}+C\frac{(\beta^+)^2}{\delta\beta^-}\sum_{B\in\mathring{\mathcal{E}}_h^s}\frac{1}{h^{\alpha}}\|[v_h]\|^2_{L^2(B)},
\end{aligned}
\end{equation*}
where the final step uses Young’s inequality. Combining the above estimates, we obtain a lower bound for $a_h(v_h,v_h)$:
\begin{equation*}
\begin{aligned}
a_h(v_h,v_h)
&\geq\sum_{K\in\mathcal{T}_h}\|\beta^\frac{1}{2}\nabla v_h\|^2_{L^2(K)}+\frac{\delta(\epsilon-1)}{2}\sum_{K\in\mathcal{T}_h^s}\|\beta^{\frac{1}{2}}\nabla v_h\|^2_{L^2(K)}\\&~~~~~~+\frac{C(\beta^+)^2(\epsilon-1)}{\delta\beta^-}\sum_{B\in\mathring{\mathcal{E}}_h^s}\frac{1}{h^{\alpha}}\|[v_h]\|^2_{L^2(B)}+\sum_{B\in\mathring{\mathcal{E}}_h^s}\frac{\sigma_B^0}{h^\alpha}\|[v_h]\|^2_{L^2(B)}\\
&\geq(1-\frac{\delta}{2}|1-\epsilon|)\sum_{K\in\mathcal{T}_h}\|\beta^\frac{1}{2}\nabla v_h\|^2_{L^2(K)}+\frac{\sigma_B^0-\frac{C(\beta^+)^2|1-\epsilon|}{\delta\beta^-}}{h^\alpha}\sum_{B\in\mathring{\mathcal{E}}_h^s}\|[v_h]\|^2_{L^2(B)}.\label{39}
\end{aligned}
\end{equation*}
If $\epsilon=0$, let $\delta=1$, and choose $\sigma_B^0\geq C\frac{(\beta^+)^2}{\beta^-}$. 
If $\epsilon=-1$, let $\delta=\frac{1}{2}$, and choose $\sigma_B^0\geq C\frac{(\beta^+)^2}{\beta^-}$.
\end{proof}

The next lemma establishes the continuity of the bilinear form  $a_h(\cdot,\cdot)$ with respect to the energy norm $\|\cdot\|_h$.

\begin{lemma}\label{lemma4}
For $a_h(\cdot,\cdot)$ defined in \eqref{3.2}, there exists a constant $C$, independent of the mesh size, the relative position of the interface, and the diffusion coefficients $\beta^{\pm}$, such that
\begin{equation}
a_h(w_h,v_h)\leq C\frac{\beta^+}{\sqrt{\beta^-}}\|w_h\|_{h}\|v_h\|_{h},\quad\forall w_h,v_h\in S^m_h(\Omega).\label{40}
\end{equation}
\end{lemma}

\begin{proof}
By the triangle inequality, we have
\begin{equation*}
\begin{aligned}
|a_h(w_h,v_h)| \leq &
\left|\sum_{K \in \mathcal{T}_h} \int_K \beta \, \nabla w_h \cdot \nabla v_h \, dX\right| +\left| \sum_{B \in \mathring{\mathcal{E}}_h^s} \int_B \{ \beta \nabla w_h \cdot \mathbf{n}_B \} [v_h] \, ds\right| \\
& +\left| \epsilon \sum_{B \in \mathring{\mathcal{E}}_h^s} \int_B \{ \beta \nabla v_h \cdot \mathbf{n}_B \} [w_h] \, ds\right|  + \left|\sum_{B \in \mathring{\mathcal{E}}_h^s} \int_B\frac{\sigma_B^0 }{ |B|^\alpha} [w_h] [v_h] \, ds\right|.\label{41}
\end{aligned}
\end{equation*}
Denote each term on the right-hand side by $Q_i$, with $i=1,2,3,4$. We estimate each term separately. 

First, applying the Cauchy-Schwarz inequality yields
\begin{equation*}
\begin{aligned}
Q_1&\leq\sum_{K \in \mathcal{T}_h}\|\sqrt{\beta}\nabla w_h\|_{L^2(K)}\|\sqrt{\beta}\nabla v_h\|_{L^2(K)}\\
&\leq \bigg(\sum_{K \in \mathcal{T}_h}\|\sqrt{\beta}\nabla w_h\|^2_{L^2(K)}\bigg)^{\frac{1}{2}}\bigg(\sum_{K \in \mathcal{T}_h}\|\sqrt{\beta}\nabla v_h\|^2_{L^2(K)}\bigg)^{\frac{1}{2}}\\
&\leq  \|w_h\|_{h}\|v_h\|_{h}.
\end{aligned}
\end{equation*}
For $Q_2$,  by the trace inequality \eqref{3.6} and the Cauchy-Schwarz inequality, we have
\begin{equation*}
\begin{aligned}
Q_2
&\leq\sum_{B\in\mathring{\mathcal{E}}_h^s}\|\{\beta\nabla w_h\cdot\mathbf{n}_B\}\|_{L^2(B)}\bigg(\frac{1}{|B|^{\alpha}}\bigg)^{\frac{1}{2}-\frac{1}{2}}\|[v_h]\|_{L^2(B)}\\
&\leq\frac{C\beta^+}{\sqrt{\beta^-}}h^{\frac{\alpha}{2}-\frac{1}{2}}\sum_{B\in\mathring{\mathcal{E}}_h^s}\bigg(\|\sqrt{\beta}\nabla w_h\|^2_{L^2(K_{B,1})}+\|\sqrt{\beta}\nabla w_h\|^2_{L^2(K_{B,2})}\bigg)^{\frac{1}{2}}\bigg(\frac{1}{h^{\alpha}}\bigg)^{\frac{1}{2}}\|[v_h]\|_{L^2(B)}\\
&\leq \frac{C\beta^+}{\sqrt{\beta^-}}\sum_{B\in\mathring{\mathcal{E}}_h^s}\bigg(\|\sqrt{\beta}\nabla w_h\|^2_{L^2(K_{B,1})}+\|\sqrt{\beta}\nabla w_h\|^2_{L^2(K_{B,2})}\bigg)^{\frac{1}{2}}\bigg(\frac{1}{h^{\alpha}}\bigg)^{\frac{1}{2}}\|[v_h]\|_{L^2(B)}\\
&\leq \frac{C\beta^+}{\sqrt{\beta^-}}\bigg(\sum_{B\in\mathring{\mathcal{E}}_h^s}\frac{1}{h^{\alpha}}\|[v_h]\|^2_{L^2(B)}\bigg)^{\frac{1}{2}}       \bigg(\sum_{K\in\mathcal{T}_h^s}\|\sqrt{\beta}\nabla w_h\|^2_{L^2(K)}\bigg)^{\frac{1}{2}}\\
&\leq \frac{C\beta^+}{\sqrt{\beta^-}} \|w_h\|_{h}\|v_h\|_{h}.
\end{aligned}
\end{equation*}
Here, we have used the fact that $0<h\leq 1$ and $\alpha\geq 1$. 

By a symmetric argument, $Q_3$ satisfies
\begin{equation*}
\begin{aligned}
Q_3
\leq \frac{C\beta^+}{\sqrt{\beta^-}} \|w_h\|_{h}\|v_h\|_{h}.
\end{aligned}
\end{equation*}

Finally, for $Q_4$, the Cauchy-Schwarz inequality yields
\begin{equation*}
\begin{aligned}
|Q_4|
&\leq \bigg(\sum_{B \in \mathring{\mathcal{E}}_h^s} \int_B\frac{\sigma_B^0 }{ |h|^\alpha}|[w_h]|^2\,ds\bigg)^{\frac{1}{2}}\bigg(\sum_{B \in \mathring{\mathcal{E}}_h^s} \int_B\frac{\sigma_B^0 }{ |h|^\alpha} |[v_h]|^2\,ds\bigg)^{\frac{1}{2}} \leq \|w_h\|_{h}\|v_h\|_{h}.
\end{aligned}
\end{equation*}
Combining the above estimates, we obtain \eqref{40}.
\end{proof}

The well-posedness of the SDG method follows from the Lax--Milgram theorem together with Lemmas \ref{lemma3} and \ref{lemma4}.
\begin{theorem}\label{thm1}
Under the conditions of Lemma~\ref{lemma3}, there exists a unique solution $u_h\in S_h^m(\Omega)$ to the SDG method \eqref{3.1}-\eqref{3.2}. 
\end{theorem}

\section{Error Estimation for the SDG Method}\label{sec4}
In this section, we derive optimal error estimates for the high-order SDG solution. Before proceeding, we introduce several operators.

On each interface element $K\in\mathcal{T}_h^i$, let $P_{h,K}:L^2(K)\to S_h^m(K)$ be the $L^2$ projection defined by
\begin{equation*}
(u-P_{h,K}u,v_h)=0,\quad\forall v_h\in S_h^m(K).
\end{equation*}

On each non-interface element $K\in\mathcal{T}_h^n$, let $I_{h,K}:C^0(K)\to S_h^m(K)$ be the Lagrange interpolant. 

For each element $K\in\mathcal{T}_h$, define the projection $\Pi_h:\mathcal{H}^{m+1}(\Omega,\Gamma;\beta)\rightarrow S_h^m(\Omega)$ by
\begin{equation}
\Pi_h u=
\left\{
\begin{aligned}
I_{h,K} u, & \quad \text{if } K\in \mathcal{T}_h^n, \\
P_{h,K} u, & \quad \text{if } K\in \mathcal{T}_h^i.\label{4.1}
\end{aligned} 
\right.
\end{equation}

Classical finite element estimates on a non-interface element $K\in\mathcal{T}_h^n$ give
\begin{equation}
|u-I_{h,K} u|_{H^i(K)}\leq C h^{m+1-i}|u|_{H^{m+1}(K)},\quad 0\leq i\leq m.\label{4.2}
\end{equation}
On an interface element $K\in \mathcal{T}_h^i$, it was shown in \cite{2025AdjeridLinMeghaichi} that
\begin{equation}
|u-P_{h,K} u|_{\mathcal{H}^i(K)}\leq C h^{m+1-i}\|u\|_{\mathcal{H}^{m+1}(K_F)},\quad 0\leq i\leq m,\label{4.3}
\end{equation}
where $K_F$ is the fictitious element in the tubular neighborhood of $\Gamma$.

\begin{lemma}\cite{2025AdjeridLinMeghaichi}\label{lemma5}
Let $K$ be an interface element and $K_F$ its fictitious element. Then $K_F$ intersects at most $49$ elements in $\mathcal{T}_h$.
\end{lemma}

The following lemma establishes the approximation properties of the high-order HIFE space $S_h^m(\Omega)$.
\begin{lemma}\label{lemma6}
Let $u\in\mathcal{H}^{m+1}(\Omega,\Gamma;\beta)$ and let $\Pi_h u\in S_h^m(\Omega)$ be the projection defined in \eqref{4.1}, then there exists a constant $C$ such that
\begin{equation}
\left| u-\Pi_h u\right|_{\mathcal{H}^i(\Omega)}\leq Ch^{m+1-i}\|u\|_{\mathcal{H}^{m+1}(\Omega)},\quad 0\leq i \leq m.\label{4.4}
\end{equation}
\end{lemma}
\begin{proof}
The result follows directly from \eqref{4.2}, \eqref{4.3}, and Lemma \ref{lemma5}.
\end{proof}

The energy-norm approximation properties of $\Pi_h u$ are summarized in the following lemma.
\begin{lemma}\label{lemma7}
Let $u\in\mathcal{H}^{m+1}(\Omega,\Gamma;\beta)$ and let $\Pi_h u$ be the projection defined in \eqref{4.1}, then
\begin{equation}
\left\| u-\Pi_h u\right\|_{h}\leq C\frac{\beta^+}{\sqrt{\beta^-}}h^{m}\|u\|_{\mathcal{H}^{m+1}(\Omega)}.\label{4.5}
\end{equation}
\end{lemma}
\begin{proof}
According to Lemma \ref{lemma6}, we have
\begin{equation}
\sum_{K\in\mathcal{T}_h}\|\sqrt{\beta}\nabla (u-\Pi_h u)\|^2_{L^2(K)}\leq C\beta^+h^{2m}\|u\|^2_{\mathcal{H}^{m+1}(\Omega)}.\label{4.6}
\end{equation}
Since $u|_{K}\in H^1(K)$ for every element $K$, the Sobolev trace theorem together with Lemma \ref{lemma6} yields
\begin{equation}
\begin{aligned}
\sum_{B\in\mathring{\mathcal{E}}_h^s}\frac{\sigma_B^0}{h^\alpha}\|[u-\Pi_h u]_B\|^2_{L^2(B)}&\leq 2\frac{\sigma_B^0}{h^\alpha}\sum_{K\in\mathcal{T}_h^i}\| (u-\Pi_h u)|_{K}\|^2_{L^2(\partial K)}\\
&\leq C\frac{\sigma_B^0}{h^\alpha}\sum_{K\in\mathcal{T}_h^i}\Big(h^{-1}\|u-\Pi_h u\|^2_{L^2(K)} +h\|\nabla (u-\Pi_h u)\|^2_{L^2(K)}  \Big)\\
&\leq C\sigma_B^0h^{2m}\|u\|^2_{\mathcal{H}^{m+1}(\Omega)}.\label{4.7}
\end{aligned}
\end{equation}
Combining \eqref{4.6}, \eqref{4.7} and \eqref{3.3} gives
\begin{equation*}
\begin{aligned}
\|u-\Pi_h u\|^2_{h}&\leq C\frac{(\beta^+)^2}{\beta^-}h^{2m}\|u\|^2_{\mathcal{H}^{m+1}(\Omega)},
\end{aligned}
\end{equation*}
where we have used the choice $\sigma_B^0=C\frac{(\beta^+)^2}{\beta^-}$. This proves the error bound \eqref{4.5}. 
\end{proof}

We are now ready to derive an {\it a priori} error bound for the SDG solution.

\begin{theorem}\label{thm2}
Assume that $u\in\mathcal{H}^{m+1}(\Omega,\Gamma;\beta)$ is the exact solution to the interface problem \eqref{1.1}-\eqref{1.3} and that $u_h\in S_h^m(\Omega)$ is the SDG solution computed with $\alpha=1$ on a rectangular mesh $\mathcal{T}_h$. Then there exists a constant $C$, independent of the mesh size, the relative position of the interface, and the diffusion coefficients $\beta^{\pm}$, such that 
\begin{equation}
\|u-u_h\|_{h}\leq C\frac{\beta^+}{\sqrt{\beta^-}}h^{m}\|u\|_{\mathcal{H}^{m+1}(\Omega)}.\label{56}
\end{equation}
\end{theorem}
\begin{proof}
Since $\mathcal{H}^{m+1}(\mathcal{T}_h,\Gamma;\beta)\subset \mathcal{H}^{m+1}(\Omega,\Gamma;\beta)$, taking $v=v_h\in \mathcal{H}^{m+1}(\mathcal{T}_h,\Gamma;\beta)$ in \eqref{3.1} gives
\begin{align*}
a_h(u,v_h)=(f,v_h),\quad\forall v_h\in \mathcal{H}^{m+1}(\mathcal{T}_h,\Gamma;\beta).
\end{align*}
By \eqref{3.1}, we have
\begin{align*}
a_h(u_h,v_h)=a_h(u,v_h),\quad\forall v_h\in \mathcal{H}^{m+1}(\mathcal{T}_h,\Gamma;\beta).\label{60}
\end{align*}
Subtracting $a_h(w_h,v_h)$ from both sides and using the linearity of  $a_h(\cdot,\cdot)$, we obtain
\begin{equation*}
a_h(u_h-w_h,v_h) = a_h(u-w_h,v_h), \quad \forall v_h, w_h \in S_h^m(\Omega). 
\end{equation*}
Taking $v_h = u_h - w_h$ and using the coercivity of $a_h(\cdot,\cdot)$, we obtain
\begin{equation}
\begin{aligned}
&\quad \kappa \, \|u_h-w_h\|_{h}^2 \\
&\le |a_h(u_h - w_h,u_h - w_h)| =|a_h(u - w_h,u_h - w_h)|\\
&\leq 
\left|\sum_{K \in \mathcal{T}_h} \int_K \beta \, \nabla (u - w_h) \cdot \nabla (u_h - w_h) \, dX\right|  +\left| \sum_{B \in \mathring{\mathcal{E}}_h^s} \int_B \{ \beta \nabla (u - w_h) \cdot \mathbf{n}_B \} [u_h - w_h] \, ds\right| \\
&\quad +\left| \epsilon \sum_{B \in \mathring{\mathcal{E}}_h^s} \int_B \{ \beta \nabla (u_h - w_h) \cdot \mathbf{n}_B \} [u - w_h] \, ds\right|  + \left|\sum_{B \in \mathring{\mathcal{E}}_h^s} \int_B\frac{\sigma_B^0 }{h} [u - w_h][u_h - w_h]  \, ds\right|\\
&=: Q_1 + Q_2 + Q_3 + Q_4.
\label{4.9}
\end{aligned}
\end{equation}

To bound $Q_1$, using Cauchy-Schwarz inequality and Young's inequality, we obtain
\begin{equation*}
\begin{aligned}
Q_1 &\leq 
\left(\sum_{K \in \mathcal{T}_h} \|\sqrt{\beta} \nabla (u - w_h)\|_{L^2(K)}^2\right)^{\frac{1}{2}} \left(\sum_{K \in \mathcal{T}_h} \|\sqrt{\beta}\nabla (u_h - w_h)\|_{L^2(K)}^2\right)^{\frac{1}{2}}\\
&\leq \frac{3}{2 \kappa}   \beta^+ \, \|\nabla (u - w_h)\|_{L^2(\Omega)}^2 
+ \frac{\kappa}{6} \sum_{K \in \mathcal{T}_h} \|\sqrt{\beta} \nabla (u_h - w_h)\|_{L^2(K)}^2 \\
&\leq C \beta^+\, \|\nabla (u - w_h)\|_{L^2(\Omega)}^2 + \frac{\kappa}{6} \| u_h - w_h \|_{h}^2. 
\end{aligned}
\end{equation*}
To bound $Q_2$, we have
\begin{equation*}
\begin{aligned}
Q_2 &\leq 
\frac{\kappa}{6} \sum_{B \in \mathring{\mathcal{E}}_h^s} \frac{\sigma_B^0}{ h} \| [u_h - w_h] \|_{L^2(B)}^2
+ C \sum_{B \in \mathring{\mathcal{E}}_h^s} \frac{h} {\sigma_B^0} \|\{ \beta \nabla (u - w_h) \cdot \mathbf{n}_B \}\|_{L^2(B)}^2\\
&\leq \frac{\kappa}{6} \|u_h-w_h\|_{h}^2+ C\sum_{B \in \mathring{\mathcal{E}}_h^s} \frac{h} {\sigma_B^0} \|\{ \beta \nabla (u - w_h) \cdot \mathbf{n}_B \}\|_{L^2(B)}^2. 
\end{aligned}
\end{equation*}
To bound $Q_3$, for each $B \in \mathring{\mathcal{E}}_h^s$, let $K_{B,i} \in \mathcal{T}_h$, $i=1,2$, be such that
$B =K_{B,1} \cap K_{B,2}.$
First, by the standard trace inequality on elements for $H^1$ functions, we have
\begin{equation}
\begin{aligned}
\| [u - w_h] \|_{L^2(B)}
\le& \| (u - w_h)|_{K_{B,1}} \|_{L^2(B)}
   + \| (u - w_h)|_{K_{B,2}} \|_{L^2(B)} \\
\le& C h^{-1/2} \left(\| u - w_h \|_{L^2(K_{B,1})}
+ h \| \nabla (u - w_h) \|_{L^2(K_{B,1})}\right) \\
&+ Ch^{-1/2} \left(\| u - w_h \|_{L^2(K_{B,2})}
+ h \| \nabla (u - w_h) \|_{L^2(K_{B,2})}\label{4.10}
\right).
\end{aligned}
\end{equation}
Then, applying the trace inequalities established in Lemma \ref{lemma2}, we obtain
\begin{equation}
\begin{aligned}
&\| \{ \beta \nabla (u_h - w_h) \cdot \mathbf{n}_B \} \|_{L^2(B)}\\
\leq&\frac{1}{2}\|  (\beta \nabla (u_h - w_h) \cdot \mathbf{n}_B)|_{K_{B,1}}  \|_{L^2(B)}
+\frac{1}{2}\|  (\beta \nabla (u_h - w_h) \cdot \mathbf{n}_B)|_{K_{B,2}}  \|_{L^2(B)}\\
\le& \frac{C\beta^+}{\sqrt{\beta^-}} h^{-1/2} \Big(
\| \sqrt{\beta} \nabla (u_h - w_h) \|_{L^2(K_{B,1})}
+ \| \sqrt{\beta}  \nabla (u_h - w_h) \|_{L^2(K_{B,2})}
\Big).\label{4.11}
\end{aligned}
\end{equation}
According to the triangle inequality, \eqref{4.10}, and \eqref{4.11}, we have
\begin{equation*}
\begin{aligned}
Q_3
&\le |\epsilon| \sum_{B \in \mathring{\mathcal{E}}_h^s}
\| \{ \beta \nabla (u_h - w_h) \cdot \mathbf{n}_B \} \|_{L^2(B)}
\| [u - w_h] \|_{L^2(B)}  \\
&\leq\frac{C\beta^+}{\sqrt{\beta^-}} h^{-1}\sum_{K\in \mathcal{T}_h}\Big(\|u-w_h\|_{L^2(K)}+h\|\nabla(u-w_h)\|_{L^2(K)}\Big)\\&~~~~\Big(
\| \sqrt{\beta} \nabla (u_h - w_h) \|_{L^2(K_{B,1})}
+ \| \sqrt{\beta}  \nabla (u_h - w_h) \|_{L^2(K_{B,2})}
\Big)\\
&\le \frac{C(\beta^+)^2}{\beta^-} h^{-2}( \| u - w_h \|_{L^2(\Omega)}^2
+  h^{2} \| \nabla (u - w_h) \|_{L^2(\Omega)}^2)
+ \frac{\kappa}{6} \| u_h - w_h \|_{h}^2.
\end{aligned}
\end{equation*}
To bound $Q_4$, we apply the standard trace inequality to obtain
\begin{equation}
\begin{aligned}
&\int_B \frac{\sigma_B^0}{ |B|^\alpha} [u - w_h][u - w_h] \, ds\\
\le &\frac{\sigma_B^0}{ h}
\Big(
\| (u - w_h)|_{K_{B,1}} \|_{L^2(B)}
+ \| (u - w_h)|_{K_{B,2}} \|_{L^2(B)}
\Big)^2  \\
\le& C\sigma_B^0 
 |K_{B,1}|^{-1}
\big(
\| u - w_h \|_{L^2(K_{B,1})}
+ h \| \nabla (u - w_h) \|_{L^2(K_{B,1})}
\big)^2  \\
&\qquad+C\sigma_B^0
 |K_{B,2}|^{-1}
\big(
\| u - w_h \|_{L^2(K_{B,2})}
+ h \| \nabla (u - w_h) \|_{L^2(K_{B,2})}
\big)^2\\
\le& C\sigma_B^0 h^{-2}
\Big(
\| u - w_h \|_{L^2(K_{B,1})}
+ h \| \nabla (u - w_h) \|_{L^2(K_{B,1})}\Big) \\
&~~+ C\sigma_B^0 h^{-2}
\Big(\| u - w_h \|_{L^2(K_{B,2})}
+ h \| \nabla (u - w_h) \|_{L^2(K_{B,2})}
\Big),\label{4.12}
\end{aligned}
\end{equation}
where we have used 
$|K_{B,i}| \leq  h^2$, for $i = 1,2$.
Then according to \eqref{4.12}, we get
\begin{equation*}
\begin{aligned}
Q_4
&\le \sum_{B \in \mathring{\mathcal{E}}_h^s}
\left(
\frac{\kappa}{6} \int_B \frac{\sigma_B^0}{ |B|^\alpha} [u_h - w_h][u_h - w_h] \, ds
+ \frac{3}{2\kappa} \int_B \frac{\sigma_B^0}{ |B|^\alpha} [u - w_h][u - w_h] \, ds
\right) \\
&\le \frac{\kappa}{6} \| u_h - w_h \|_{h}^2
+ \frac{3}{2\kappa}
\sum_{B \in \mathring{\mathcal{E}}_h^s}
\int_B \frac{\sigma_B^0} {|B|^\alpha} [u - w_h][u - w_h] \, ds  \\
&\le \frac{\kappa}{6} \| u_h - w_h \|_{h}^2
+ C \sigma_B^0h^{-2}\Big( \| u - w_h \|_{L^2(\Omega)}^2
+ h^{2} \| \nabla (u - w_h) \|_{L^2(\Omega)}^2\Big).
\end{aligned}
\end{equation*}
Substituting the bounds for $Q_i$, $i=1,2,3,4$, into \eqref{4.9}, we obtain
\begin{equation}
\begin{aligned}
\| u_h - w_h \|_{h}^2
\le\;&
C\beta^+ \| \nabla (u - w_h) \|_{L^2(\Omega)}^2
+ C \sum_{B \in \mathring{\mathcal{E}}_h^s}
\frac{h}{\sigma_B^0}
\| \{ \beta \nabla (u - w_h) \cdot \mathbf{n}_B \} \|_{L^2(B)}^2 \\
&+ \frac{C(\beta^+)^2}{\beta^-} h^{-2} \Big(\| u - w_h \|_{L^2(\Omega)}^2
+  h^{2} \| \nabla (u - w_h) \|_{L^2(\Omega)}^2\Big)  \\
&+ C\sigma_B^0 h^{-2} \Big(\| u - w_h \|_{L^2(\Omega)}^2
+  h^{2} \| \nabla (u - w_h) \|_{L^2(\Omega)}^2\Big).\label{4.13}
\end{aligned}
\end{equation}
Setting $w_h = \Pi_h u$ in \eqref{4.13} yields
\begin{equation}
\begin{aligned}
\| u_h -  \Pi_h u \|_{h}^2
\le\;&
C\beta^+ \| \nabla (u -  \Pi_h u) \|_{L^2(\Omega)}^2
+ C \sum_{B \in \mathring{\mathcal{E}}_h^s}
\frac{h}{\sigma_B^0}
\| \{ \beta \nabla (u -  \Pi_h u) \cdot \mathbf{n}_B \} \|_{L^2(B)}^2 \\
&+ \frac{C(\beta^+)^2}{\beta^-} h^{-2} \Big(\| u -  \Pi_h u \|_{L^2(\Omega)}^2
+  h^{2} \| \nabla (u -  \Pi_h u) \|_{L^2(\Omega)}^2\Big)  \\
&+ C \sigma_B^0h^{-2} \Big(\| u -  \Pi_h u \|_{L^2(\Omega)}^2
+  h^{2} \| \nabla (u -  \Pi_h u) \|_{L^2(\Omega)}^2\Big).\label{4.14}
\end{aligned}
\end{equation}
By the standard trace inequality and Lemma \ref{lemma6}, we have
\begin{equation*}
\begin{aligned}
&\sum_{B\in\mathring{\mathcal{E}}_h^s}\frac{h}{\sigma_B^0}\|\{\beta \nabla (u -  \Pi_h u)\cdot \mathbf{n}\}_B\|^2_{L^2(B)}\\\leq &\frac{(\beta^+)^2}{\sigma_B^0}h\sum_{K\in\mathcal{T}_h}\| (\nabla (u -  \Pi_h u)\cdot \mathbf{n})|_{K}\|^2_{L^2(\partial K)}\\
\leq &C\frac{(\beta^+)^2}{\sigma_B^0}h\sum_{K\in\mathcal{T}_h}\Big(h^{-1}\|\nabla (u -  \Pi_h u)\|^2_{L^2(K)} +h\|\nabla^2 (u -  \Pi_h u)\|^2_{L^2(K)}  \Big)\\
\leq &C\frac{(\beta^+)^2}{\sigma_B^0}h^{2m}\|u\|^2_{\mathcal{H}^{m+1}(\Omega)}.
\end{aligned}
\end{equation*}
Then by \eqref{4.14} and Lemma \ref{lemma6}, we have 
\begin{equation}
\begin{aligned}
\| u_h -  \Pi_h u \|_{h}^2
\le\;&
C\beta^+ h^{2m}\|u\|^2_{\mathcal{H}^{m+1}(\Omega)}
+ C\frac{(\beta^+)^2}{\sigma_B^0}h^{2m}\|u\|^2_{\mathcal{H}^{m+1}(\Omega)}  \\
&~+ \frac{C(\beta^+)^2}{\beta^-}  h^{2m}\|u\|^2_{\mathcal{H}^{m+1} (\Omega)}+ C \sigma_B^0 h^{2m}\|u\|^2_{\mathcal{H}^{m+1} (\Omega)}\\
\leq &\frac{C(\beta^+)^2}{\beta^-}  h^{2m}\|u\|^2_{\mathcal{H}^{m+1} (\Omega)}.\label{73}
\end{aligned}
\end{equation}
By \eqref{73} and Lemma \ref{lemma7}, we have
 \begin{equation}
 \begin{aligned}
\|u-u_h\|_{h}&=  \| u-\Pi_h u +\Pi_h u-u_h \|_{h}\\
&\leq \| u-\Pi_h u\|_{h}+\|u_h-\Pi_h u\|_{h}\leq C\frac{\beta^+}{\sqrt{\beta^-}}h^m\|u\|_{\mathcal{H}^{m+1} (\Omega)}.\label{74}
\end{aligned}
\end{equation}
\end{proof}

Using a standard duality argument, we obtain an optimal error estimate in the $L^2$ norm.
\begin{theorem}\label{thm3}
Under the conditions of Theorem \ref{thm2}, there exists a constant $C$,  independent of the mesh size, the relative position of the interface, and the diffusion coefficients $\beta^{\pm}$, such that
\begin{equation}\label{eq: L2}
\|u-u_h\|_{L^2(\Omega)}\leq C\frac{(\beta^+)^3}{(\beta^-)^{\frac{5}{2}}}h^{m+1}\|u\|_{\mathcal{H}^{m+1}(\Omega)}.
\end{equation}
\end{theorem}
\begin{proof}
 Let $w\in \mathcal{H}^{2}(\Omega,\Gamma;\beta)$ be the solution of the auxiliary problem \eqref{1.1}-\eqref{1.3} with $f$ at the right-hand side replaced by $u-u_h$. Then we have
\begin{equation}
\|u-u_h\|_{L^2(\Omega)}^2=a_h(w,u-u_h).\label{4.18}
\end{equation}
Let $\Pi_h w$ be the projection defined in \eqref{4.1}. Since $\Pi_h w\in S_h^m(\Omega)$, \eqref{3.1} gives
\[a_h(\Pi_h w,u-u_h)=0
\]
 and therefore
\[a_h(w,u-u_h)=a_h(w-\Pi_h w,u-u_h).\] 
Then, by \eqref{4.18} and the continuity of $a_h(\cdot,\cdot)$, we obtain
\begin{equation*}
\|u-u_h\|_{L^2(\Omega)}^2=a_h(w-\Pi_h w,u-u_h)\leq  \frac{C\beta^+}{\sqrt{\beta^-}}\|w-\Pi_h w\|_{h}\|u-u_h\|_{h}.
\end{equation*}
According to Lemma \ref{lemma7} and the elliptic regularity
\begin{equation}
\beta^- \sum_{k=1}^{2} |w^-|_{H^k(\Omega^-)}
+
\beta^+ \sum_{k=1}^{2} |w^+|_{H^k(\Omega^+)}
\;\leq C\;
\|u-u_h\|_{L^2(\Omega)},\label{4.20}
\end{equation}
we have
\begin{equation}
\|w-\Pi_h w\|_{h}\leq C\frac{\beta^+}{\sqrt{\beta^-}}h\|w\|_{\mathcal{H}^{2}(\Omega)}\leq C\frac{\beta^+}{(\beta^-)^{\frac{3}{2}}}h\|u-u_h\|_{L^2(\Omega)}.\label{4.19}
\end{equation}
Then by Theorem \ref{thm2} and \eqref{4.19}, we obtain 
\begin{equation*}
\|u-u_h\|_{L^2(\Omega)}^2\leq \frac{C\beta^+}{\sqrt{\beta^-}}\cdot\Big(C\frac{\beta^+}{(\beta^-)^{\frac{3}{2}}}h\|u-u_h\|_{L^2(\Omega)}\Big)\cdot\Big(C\frac{\beta^+}{\sqrt{\beta^-}}h^m\|u\|_{\mathcal{H}^{m+1}(\Omega)} \Big).
\end{equation*}
Therefore,
\begin{equation*}
\|u-u_h\|_{L^2(\Omega)}\leq C\frac{(\beta^+)^3}{(\beta^-)^{\frac{5}{2}}}h^{m+1}\|u\|_{\mathcal{H}^{m+1}(\Omega)}.
\end{equation*}
\end{proof}

\begin{remark}
The dependence of $L^2$ error bound \eqref{eq: L2} on $\beta$ is likely not sharp; the numerical results in Section \ref{sec5} indicate substantially milder $\beta$-dependence in practice.
\end{remark}


\section{Numerical Examples}\label{sec5}
In this section, we present several numerical examples that demonstrate the performance of the SDG method developed in Sections \ref{sec3}-\ref{sec4}. 

We consider three examples with different features. The first example  has a radial solution on a circular interface. We use this example to quantify the DoF savings of the SDG method comparing to the DG and CG methods. We also verify the approximation properties of the new high-order HIFE spaces and compare  the convergence behavior of the SDG, DG and PPIFE solutions. The second example considers a star-shaped interface with variable curvature, which tests robustness of the SDG method under geometric complexity. 
The third example has a non-radial solution on a circular interface, aiming to test the ability of the SDG method to handle more general solutions. In all examples, the meshes are uniform Cartesian rectangular meshes with $N$ subdivisions per direction. We report errors in the $L^2$ and semi-$H^1$ norms.

We refer to the bilinear form \eqref{3.2} with $\epsilon=-1$ as symmetric (S) and the one with $\epsilon=1$ as nonsymmetric (N), so that the methods to be compared are S/N-SDG, S/N-DG, and S/N-PPIFE. We take $\alpha = 1$ in all schemes, consistent with the error analysis of Section 4, except for the N-SDG method, where we set $\alpha = 1$ when $m$ is odd and $\alpha = 3$ when $m$ is even, in agreement with the choice of $\alpha$
in the DG method, see Sections 2.8 and 2.10 of \cite{2008Riviere}.

\subsection{Radial solution on a circular interface}
\label{subsec:Radial_solution}
We take $\Omega=(-1,1)^2$ and the level-set
function $\psi(x,y)=\sqrt{x^2+y^2}-r_0$, which defines a circular
interface $\Gamma=\{\psi=0\}$ separating
$\Omega^-=\{\psi<0\}$ from $\Omega^+=\{\psi>0\}$. 
The exact solution is
\begin{equation}\label{eq: 5.1}
u(x,y)=
\begin{cases}
\displaystyle \frac{r^\nu}{\beta^-} & (x,y)\in \Omega^-,\\
\displaystyle \frac{r^\nu}{\beta^+}
+\left(\frac{1}{\beta^-}-\frac{1}{\beta^+}\right)r_0^{\nu} & (x,y)\in
\Omega^+,\\
\end{cases}
\end{equation}
where $r=\sqrt{x^2+y^2}$ and $\nu=5$. 

\subsubsection{Degrees of Freedom Test}
Our first comparison focuses on the computational advantage of the SDG method. 
 Let $DoF_D,DoF_S,DoF_C$ denote the degrees of freedom of the DG, SDG, and CG methods, respectively. A direct count yields, as $N\to\infty$,
 \begin{equation*}
\begin{aligned}
&\frac{DoF_{D}}{DoF_{S}}= \frac{(m+1)^2N^2}{(mN+1)^2+4m|\mathcal{T}_h^i|-(m-1)|\mathring{\mathcal{E}}_h^i|}= \frac{(m+1)^2N^2}{(mN+1)^2+(3m+1)|\mathcal{T}_h^i|}\to \frac{(m+1)^2}{m^2},\\
&\frac{DoF_{C}}{DoF_{S}}=  \frac{(mN+1)^2}{(mN+1)^2+(3m+1)|\mathcal{T}_h^i|}\to 1.
\end{aligned}
\end{equation*} 
where we have used the relation $|\mathcal{T}_h^i|=|\mathring{\mathcal{E}}_h^i|$ that holds for rectangular meshes with closed interface curves.  The SDG count is therefore asymptotically smaller than the DG count by a factor of $\frac{(m+1)^2}{m^2}$, while remaining asymptotically equivalent to the CG count. 
 Figure~\ref{fig3} and Table \ref{tab3}  confirm these ratios across a sequence of uniform $N\times N$ Cartesian meshes for $N \in [5, 640]$ for $m\le 4$.  

These savings matter most at low order, where DG is most expensive relative to CG, and the SDG construction captures essentially all of the CG savings without sacrificing the local flexibility of DG on interface elements.

\begin{figure}[h]
    \centering

    \begin{subfigure}{0.48\textwidth}
        \centering
        \includegraphics[width=\textwidth]{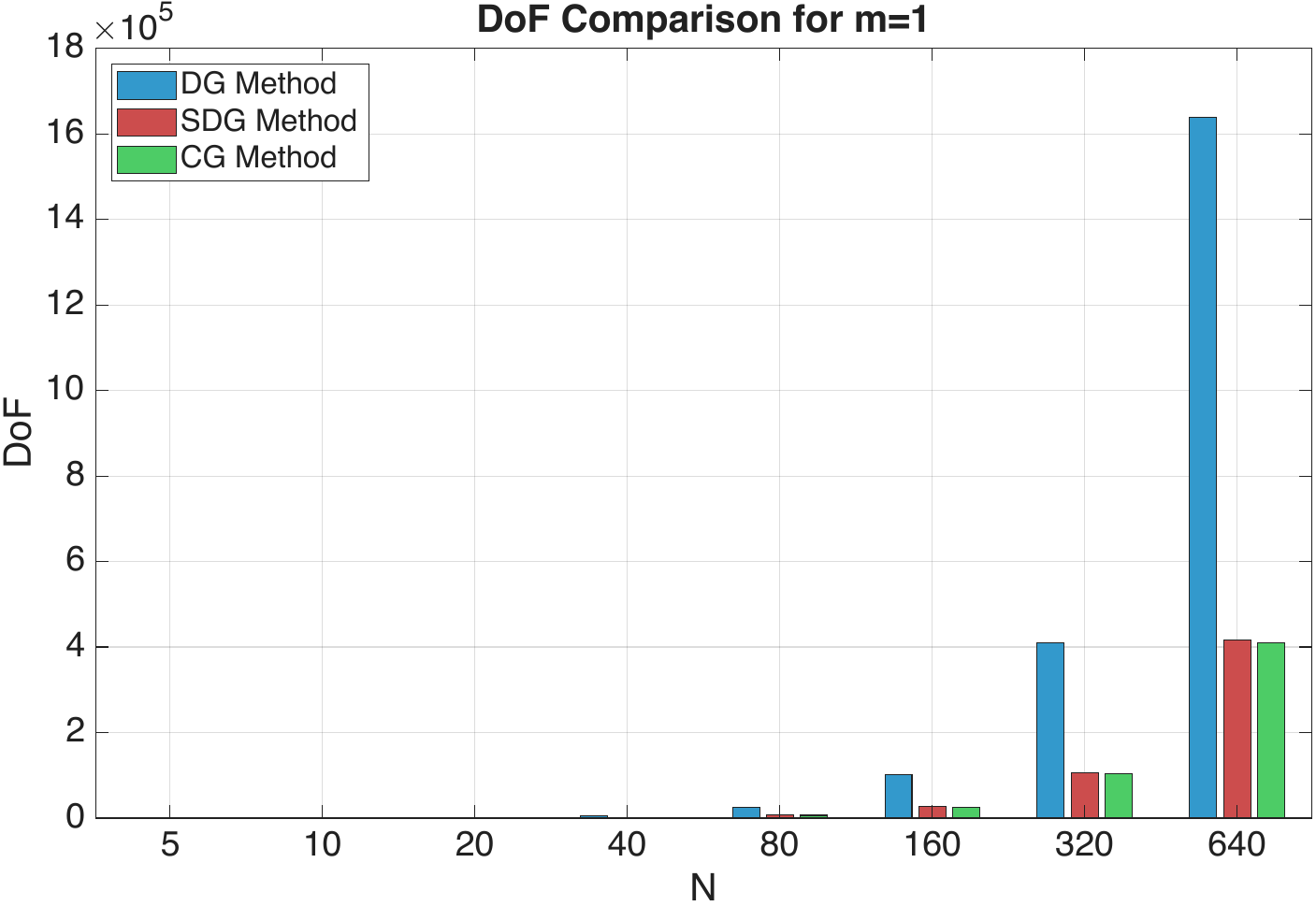}
        \caption{$m=1$}
    \end{subfigure}~~
    \begin{subfigure}{0.48\textwidth}
        \centering
        \includegraphics[width=\textwidth]{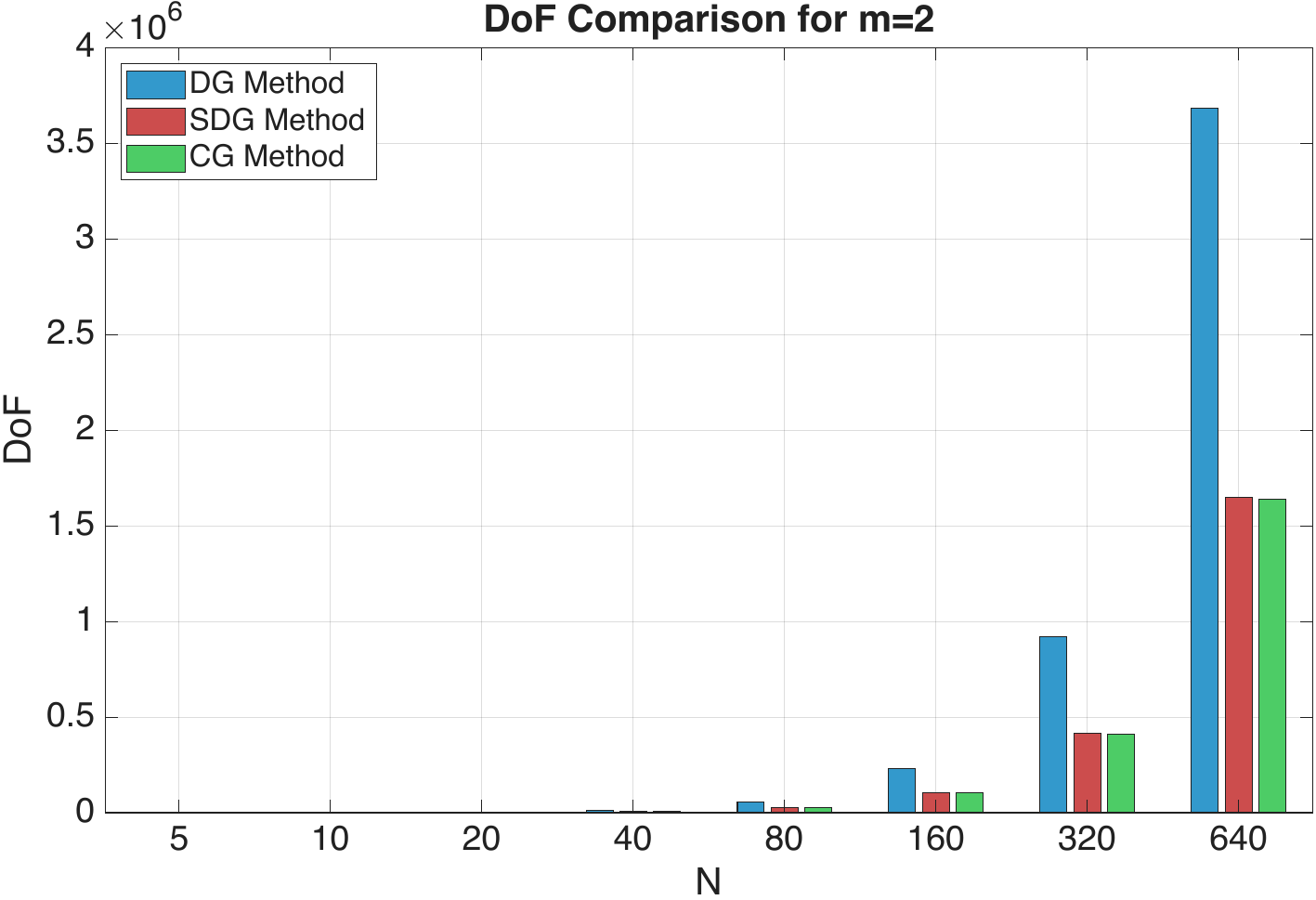}
        \caption{$m=2$}
    \end{subfigure}

    \caption{Comparison of global DoF counts for the DG, SDG, and CG methods for $m=1,2$.}
    \label{fig3}
\end{figure}

\begin{table}[h]
\centering
\begin{tabular}{cccccccccc c}
\toprule
& & & \multicolumn{2}{c}{$m=1$} & \multicolumn{2}{c}{$m=2$} & \multicolumn{2}{c}{$m=3$}& \multicolumn{2}{c}{$m=4$}\\
\cmidrule(lr){4-5} \cmidrule(lr){6-7}\cmidrule(lr){8-9}\cmidrule(lr){10-11}
$N$ & $|\mathcal{T}_h^i|$ & $|\mathring{\mathcal{E}}_h^i|$
& $\frac{DoF_{D}}{DoF_{S}}$  
& $\frac{DoF_C}{DoF_{S}}$ 
& $\frac{DoF_{D}}{DoF_{S}}$  
& $\frac{DoF_C}{DoF_{S}}$
& $\frac{DoF_{D}}{DoF_{S}}$  
& $\frac{DoF_C}{DoF_{S}}$ 
& $\frac{DoF_{D}}{DoF_{S}}$  
& $\frac{DoF_C}{DoF_{S}}$  \\
\midrule
5   & 8  & 8  & 1.4706 & 0.5294 & 1.2712 & 0.6836 & 1.1905&0.7619 & 1.1468&0.8092\\
10  & 20 &20  & 1.9900 & 0.6020 & 1.5491 & 0.7590 &1.3781 &0.8277 & 1.2880&0.8660\\
20  & 44 & 44  & 2.5932 & 0.7147 & 1.8100 & 0.8451 &1.5381 &0.8943 &1.3872 &0.9101\\
40  & 92 &92  & 3.1235 & 0.8204 & 1.9986 & 0.9106 & 1.6451 &0.9409 &1.4425&0.9348\\
80  & 188 &188 & 3.5006 & 0.8972 & 2.1148 & 0.9517 &1.7078 &0.9686 &1.4699 &0.9466\\
160 & 372 &372 & 3.7360 & 0.9457 & 2.1809 & 0.9754 &1.7424 &0.9842 &1.4833 &0.9523\\
320 & 740 &740 & 3.8641 & 0.9721 & 2.2151 & 0.9875 &1.7600 &0.9921 &1.4951 &0.9584\\
640 & 1476& 1476 & 3.9310 & 0.9858 & 2.2324 & 0.9937 &1.7689 &0.9960 & 1.5003&0.9620\\
\bottomrule
\end{tabular}
\caption{Ratios of DoF counts among DG, SDG, and CG methods for $m=1,2,3,4$.}
\label{tab3}
\end{table}

\subsubsection{Approximation capability of the high-order HIFE space}
We next verify approximation capability of the HIFE space by computing the projection error $u-\Pi_h u$ in the $L^2$ and semi-$H^1$ norms. 

Figures \ref{fig4} and \ref{fig5} show the results for $m = 1$ and 
$m = 2$, respectively, with both a moderate jump $\beta:= (\beta^-,\beta^+) =(1,10)$ and a large jump $\beta= (1,1000)$. In all cases the errors decay at the optimal rates
in the $L^2$ and the semi-$H^1$ norms in agreement with the analysis in Lemmas \ref{lemma6} and \ref{lemma7}.

\begin{figure}[!htb]
    \centering

    \begin{subfigure}{0.48\textwidth}
        \centering
        \includegraphics[width=\textwidth]{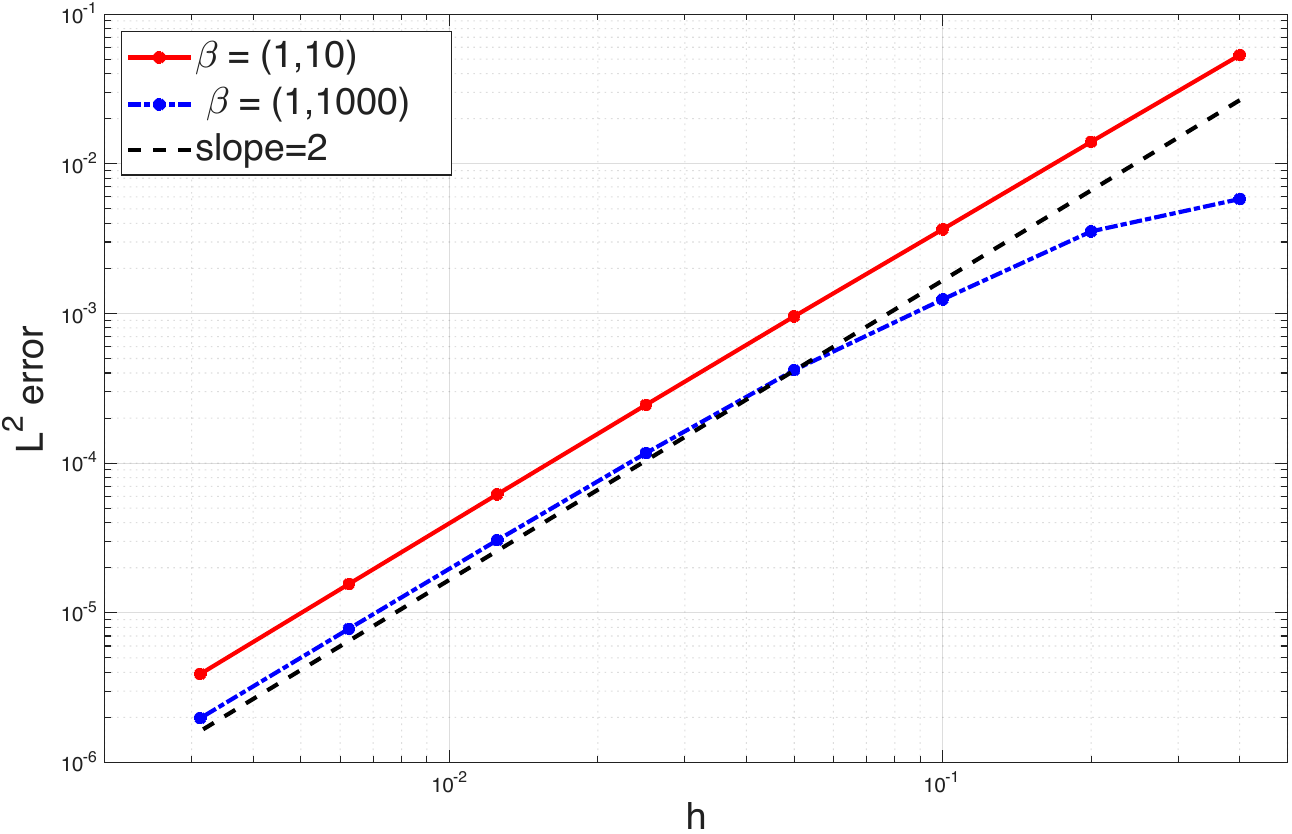}
        \caption{$L^2$ Errors }
    \end{subfigure}
    ~~
    \begin{subfigure}{0.48\textwidth}
        \centering
        \includegraphics[width=\textwidth]{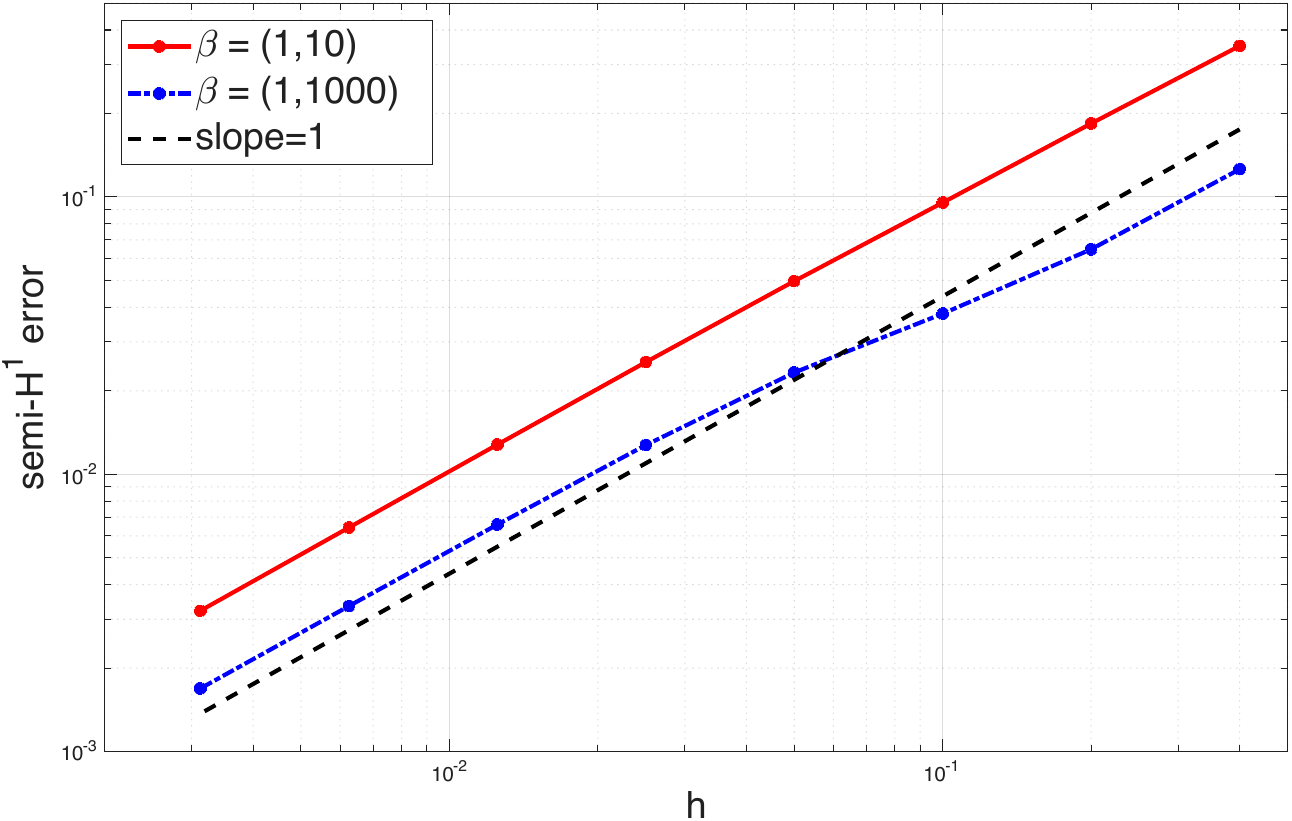}
        \caption{Semi-$H^1$ Errors}
    \end{subfigure}

    \caption{Projection errors and convergence rates in the  $L^2$ and semi-$H^1$ norms for $\beta=(1,10)$ and $\beta=(1,1000)$, $m=1$.}
    \label{fig4}
\end{figure}

\begin{figure}[!htb]
    \centering

    \begin{subfigure}{0.48\textwidth}
        \centering
        \includegraphics[width=\textwidth]{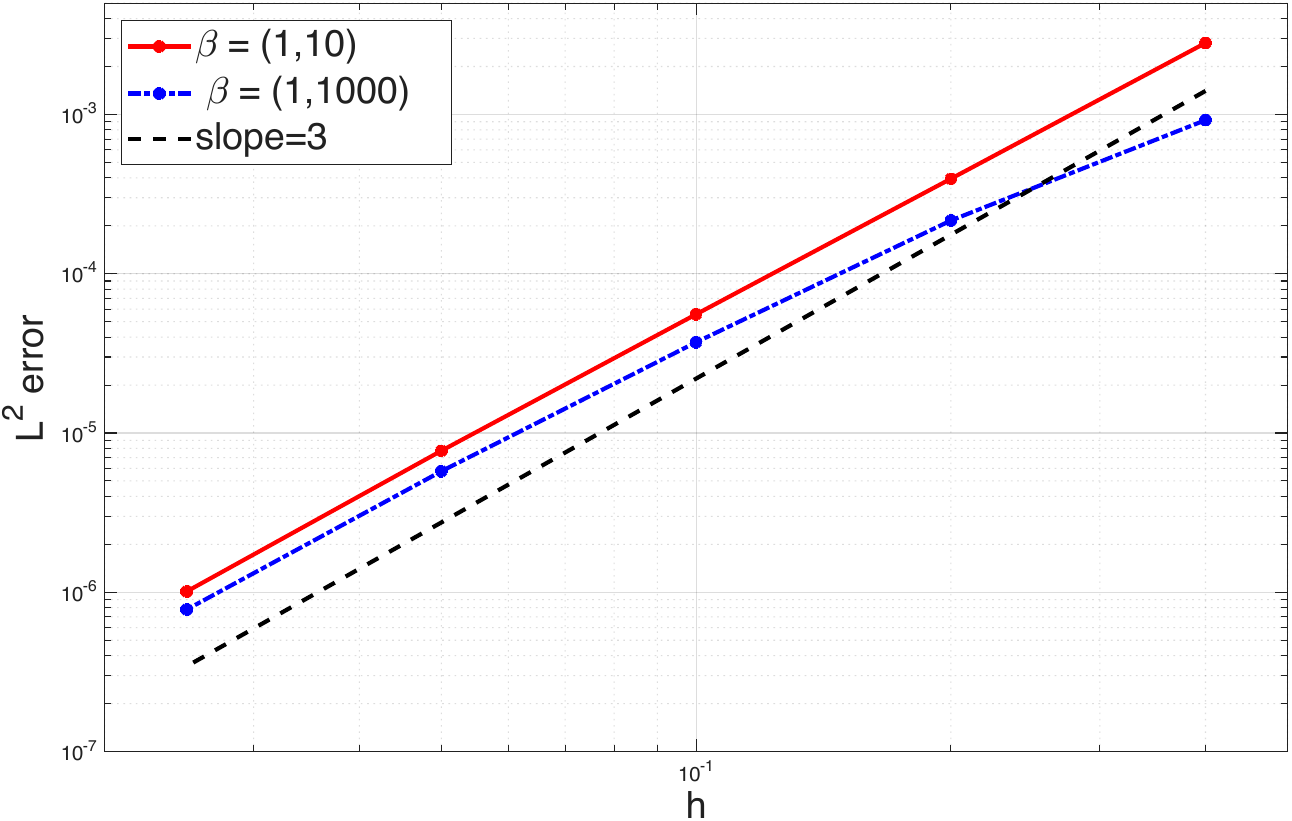}
        \caption{$L^2$ Errors }
    \end{subfigure}
    ~~
    \begin{subfigure}{0.48\textwidth}
        \centering
        \includegraphics[width=\textwidth]{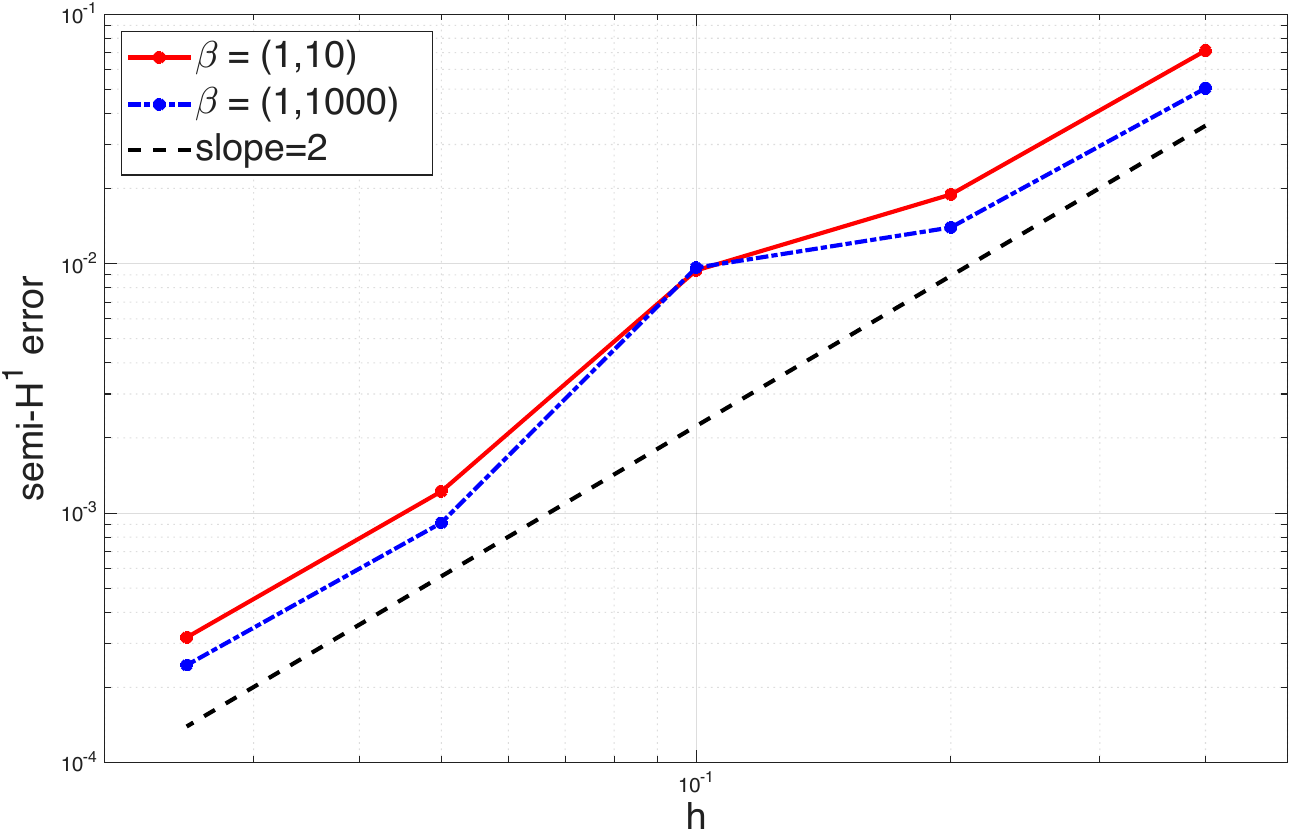}
        \caption{Semi-$H^1$ Errors}
    \end{subfigure}

    \caption{Projection errors and convergence rates in the  $L^2$ and semi-$H^1$ norms for $\beta=(1,10)$ and $\beta=(1,1000)$, $m=2$.}
    \label{fig5}
\end{figure}

\begin{figure}[!htb]
    \centering

    \begin{subfigure}{0.48\textwidth}
        \centering
        \includegraphics[width=\textwidth]{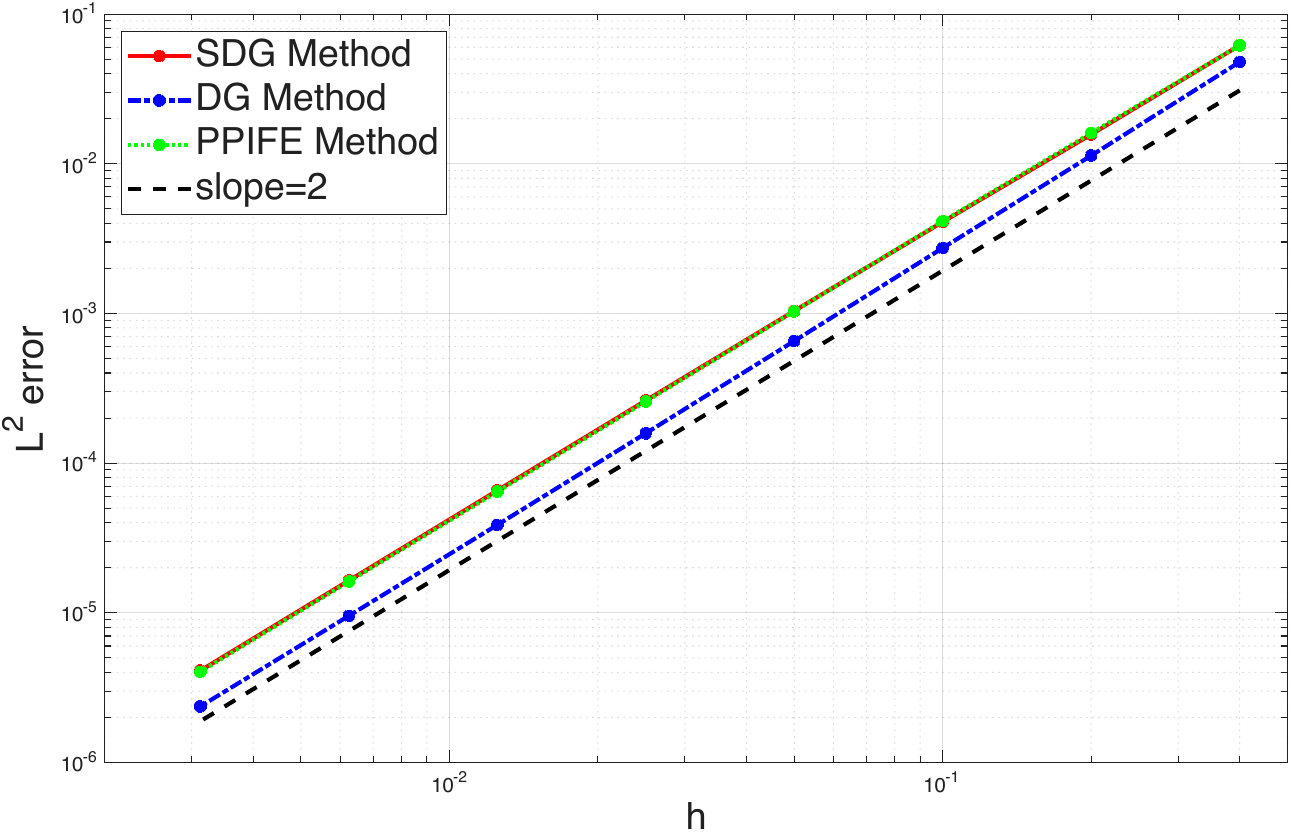}
        \caption{$L^2$ Errors }
    \end{subfigure}
    ~~
    \begin{subfigure}{0.48\textwidth}
        \centering
        \includegraphics[width=\textwidth]{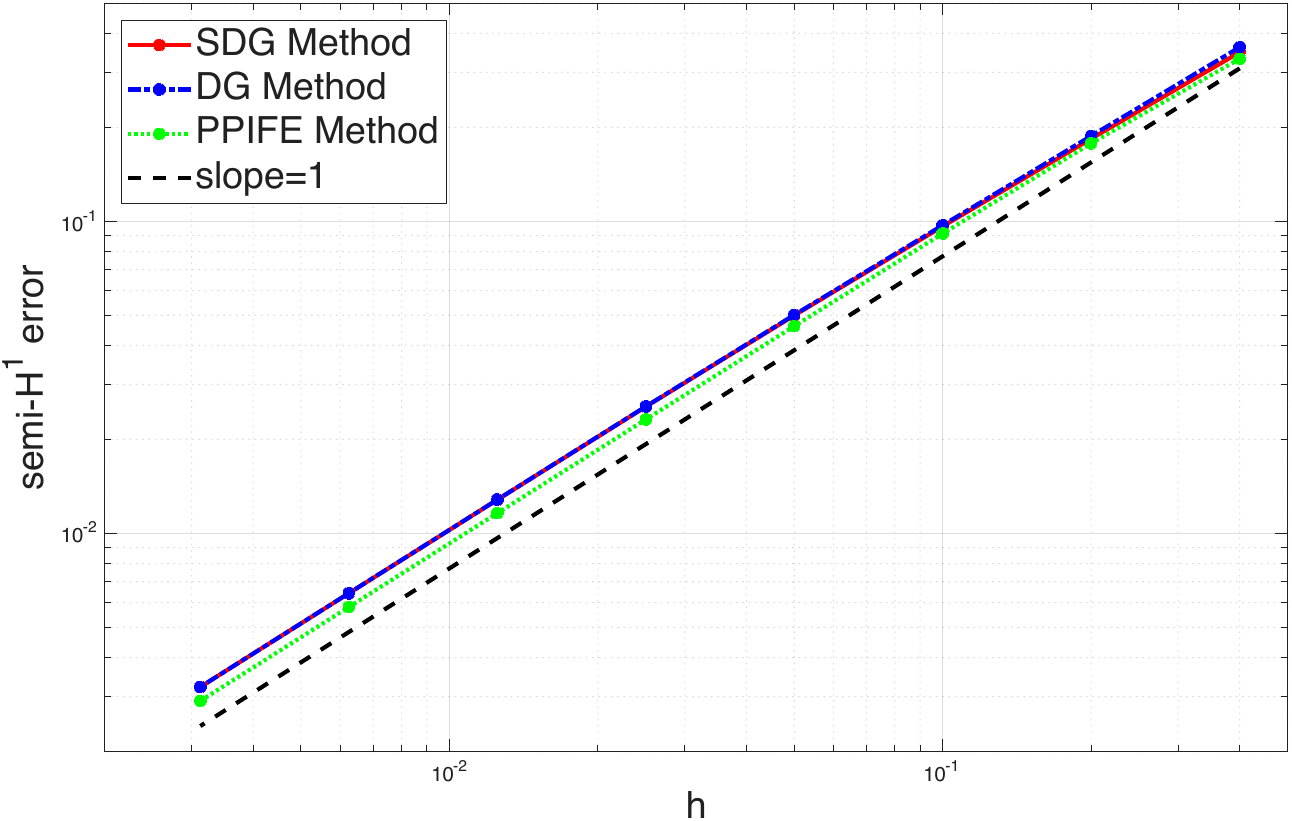}
        \caption{Semi-$H^1$ Errors}
    \end{subfigure}

    \caption{$L^2$ and semi-$H^1$ errors of the symmetric SDG,  DG, and PPIFE methods with $\beta=(1,10)$ , and $\sigma_B^0=4\max\{\beta^-, \beta^+\}$, for $m=1$.}
    \label{fig6}
\end{figure}

\begin{figure}[!htb]
    \centering

    \begin{subfigure}{0.48\textwidth}
        \centering
        \includegraphics[width=\textwidth]{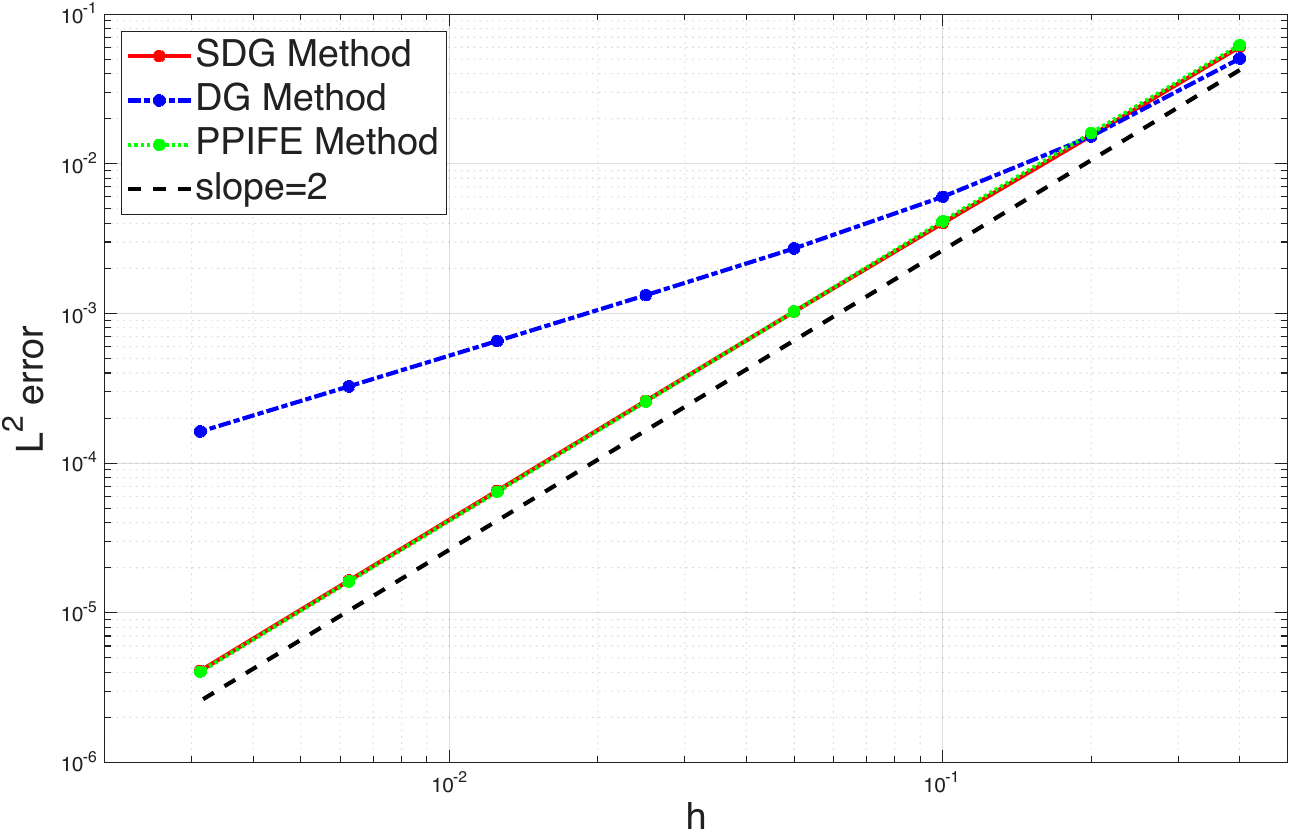}
        \caption{$L^2$ errors with $\beta=(1,10).$}
    \end{subfigure}
    ~~
    \begin{subfigure}{0.48\textwidth}
        \centering
        \includegraphics[width=\textwidth]{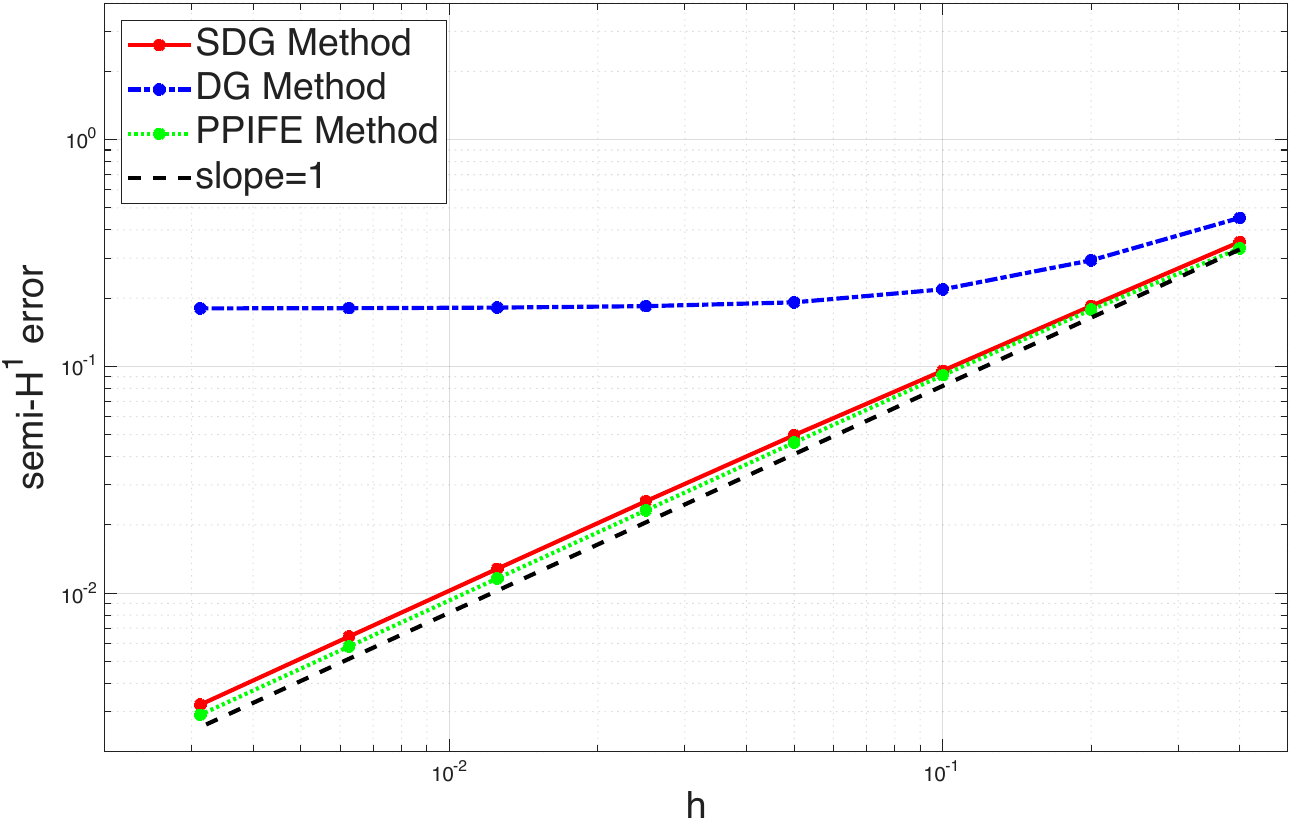}
        \caption{Semi-$H^1$ errors with $\beta=(1,10).$}
    \end{subfigure}

    \vspace{0.5em}

    \begin{subfigure}{0.48\textwidth}
        \centering
        \includegraphics[width=\textwidth]{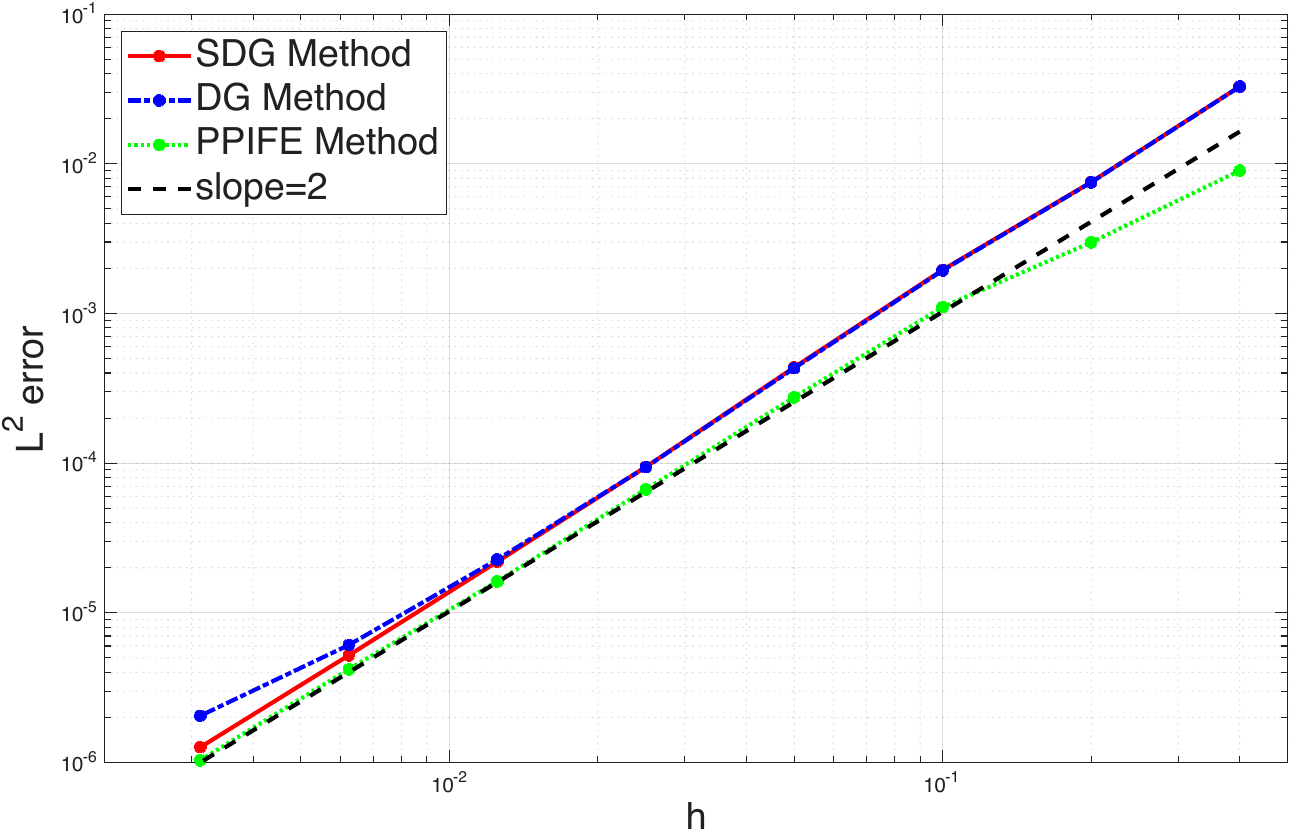}
        \caption{$L^2$ errors with $\beta=(1,1000).$}
    \end{subfigure}
    ~~
    \begin{subfigure}{0.48\textwidth}
        \centering
        \includegraphics[width=\textwidth]{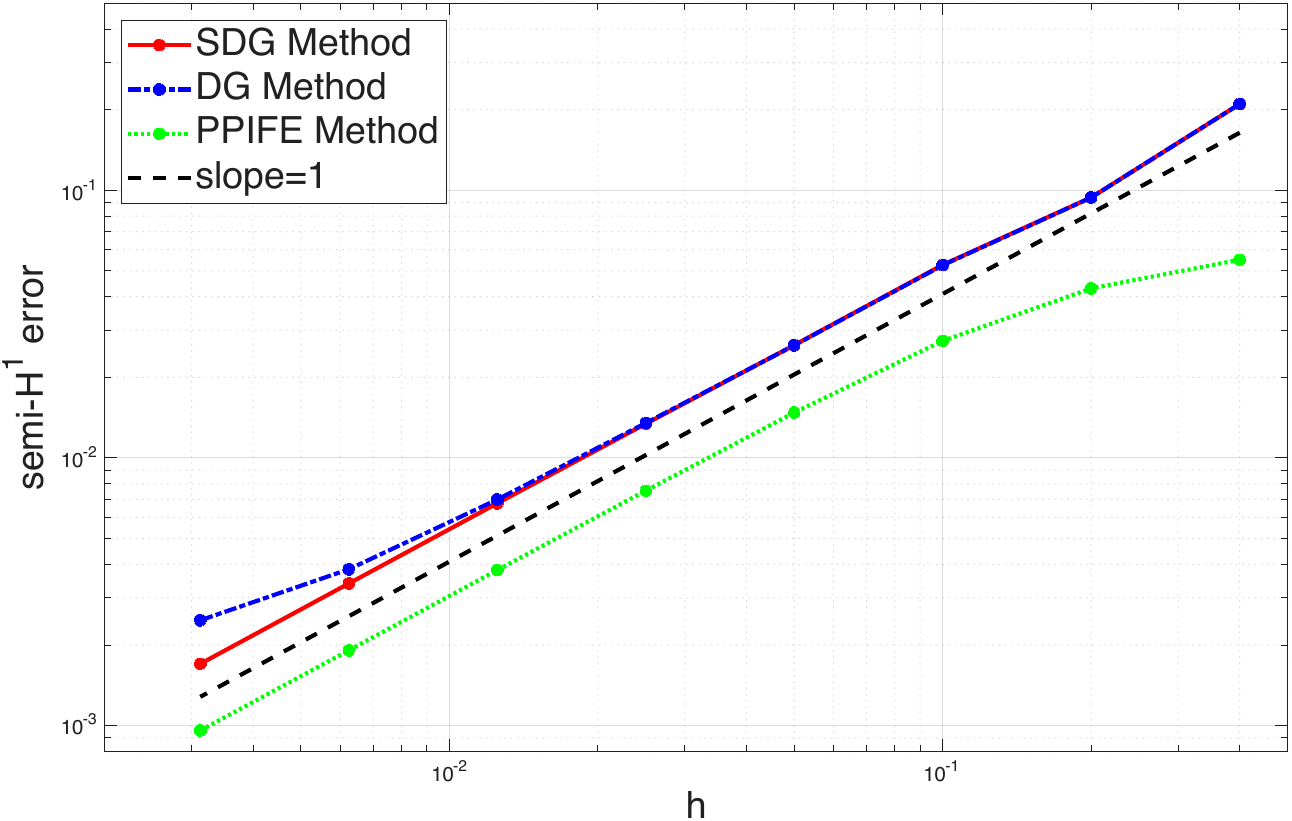}
        \caption{Semi-$H^1$ errors with $\beta=(1,1000).$}
    \end{subfigure}

    \caption{$L^2$ and semi-$H^1$ errors of the symmetric SDG, DG, and PPIFE methods with $\beta=(1,10)$ and $\beta=(1,1000)$ and $\sigma_B^0=\max(\beta^+,
    \beta^-)$ for $m=1$.}
    \label{fig7}
\end{figure} 
\subsubsection{Convergence Test}
We now compute the discrete solution errors for the six methods: S/N-SDG, S/N-DG, and S/N-PPIFE. Two penalty choices are considered:
\begin{enumerate}
     \item {\bf Large penalty}, $\sigma_B^0=4\max\{\beta^-, \beta^+\}$, following the choice in [1]. Figure \ref{fig6} shows that the S-SDG, S-DG, and S-PPIFE methods all attain the optimal rates in $L^2$ and semi-$H^1$ norms for $m = 1$ at $\beta = (1, 10)$. This confirms that the SDG method achieves the same optimal rates as the DG-IFE and PPIFE methods when the penalty is taken sufficiently large.
    \item {\bf Reduced penalty}, $\sigma_B^0=\max\{\beta^-, \beta^+\}$. The convergence behavior in this parameter setting is more revealing and is shown in Figure \ref{fig7}. For the moderate jump $\beta=(1,10)$ (panels (a)–(b)), the S-SDG and S-PPIFE methods achieve optimal rates, but the S-DG method fails to attain the optimal rate in either norm. For the large jump $\beta=(1,1000)$ (panels (c)–(d)), all three methods recover optimal rates.
\end{enumerate}



 
Figure \ref{fig8} shows that the nonsymmetric variants (N-SDG, N-DG, N-PPIFE) achieve optimal convergence in both norms for $\beta=(1,10)$ and 
$\beta=(1,1000)$ at $\sigma_B^0=\max\{\beta^-, \beta^+\}$, consistent with the unconditional stability of nonsymmetric DG-type methods. Figure \ref{fig9} confirms that the S-SDG and N-SDG methods both attain optimal convergence at $m=2$ across both jump ratios. These results validate the theoretical convergence rates of the SDG scheme across the various parameters.

\begin{remark} The difference in Figure \ref{fig7} (a)–(b) between the symmetric DG method and its SDG/PPIFE counterparts under reduced penalty is a distinctive numerical finding. The symmetric DG bilinear form penalizes jumps on every interior edge $B\in \mathring{\mathcal{E}}_h$. When $\sigma_B^0$ is reduced, the resulting stabilization becomes insufficient on the full edge set, and the coercivity degrades. In contrast, the symmetric SDG method imposes its penalty only on the edges $B \in \mathring{\mathcal{E}}_h^s$ of interface elements, and the symmetric PPIFE method only on the interface edges $B \in \mathring{\mathcal{E}}_h^i$, both are much smaller subsets compared to $\mathring{\mathcal{E}}_h$, so a smaller stabilization parameter suffices. The recovery at $\beta = (1, 1000)$ reflects the fact that the effective penalty scales with $\max\{\beta^-, \beta^+\}=1000$ and is therefore already large in absolute terms. A systematic analysis of the dependence of optimal convergence on penalty parameters for symmetric IPDG-type methods would be of independent interest and is left to future work.
\end{remark}

\begin{figure}[htb]
    \centering

    \begin{subfigure}{0.48\textwidth}
        \centering
        \includegraphics[width=\textwidth]{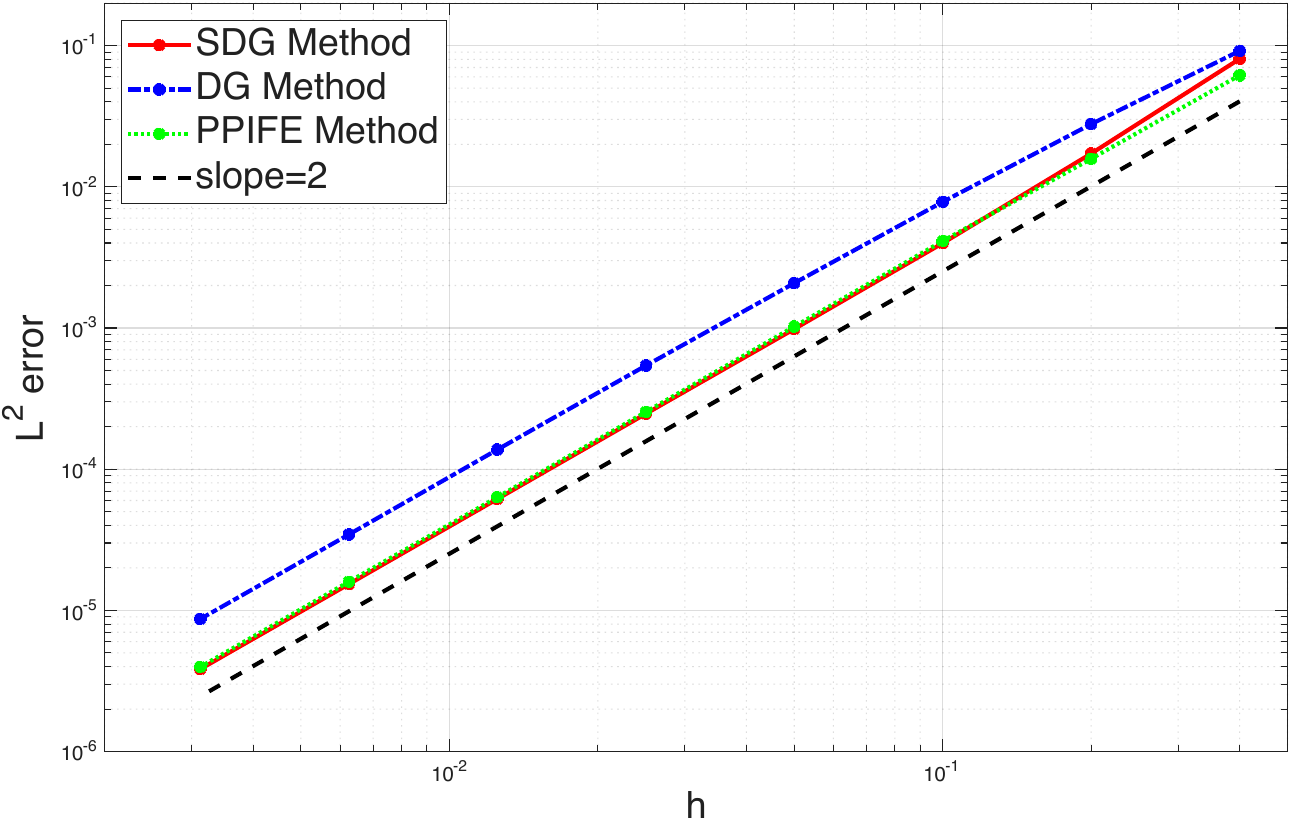}
        \caption{$L^2$ errors with $\beta=(1,10).$}
    \end{subfigure}
    ~~
    \begin{subfigure}{0.48\textwidth}
        \centering
        \includegraphics[width=\textwidth]{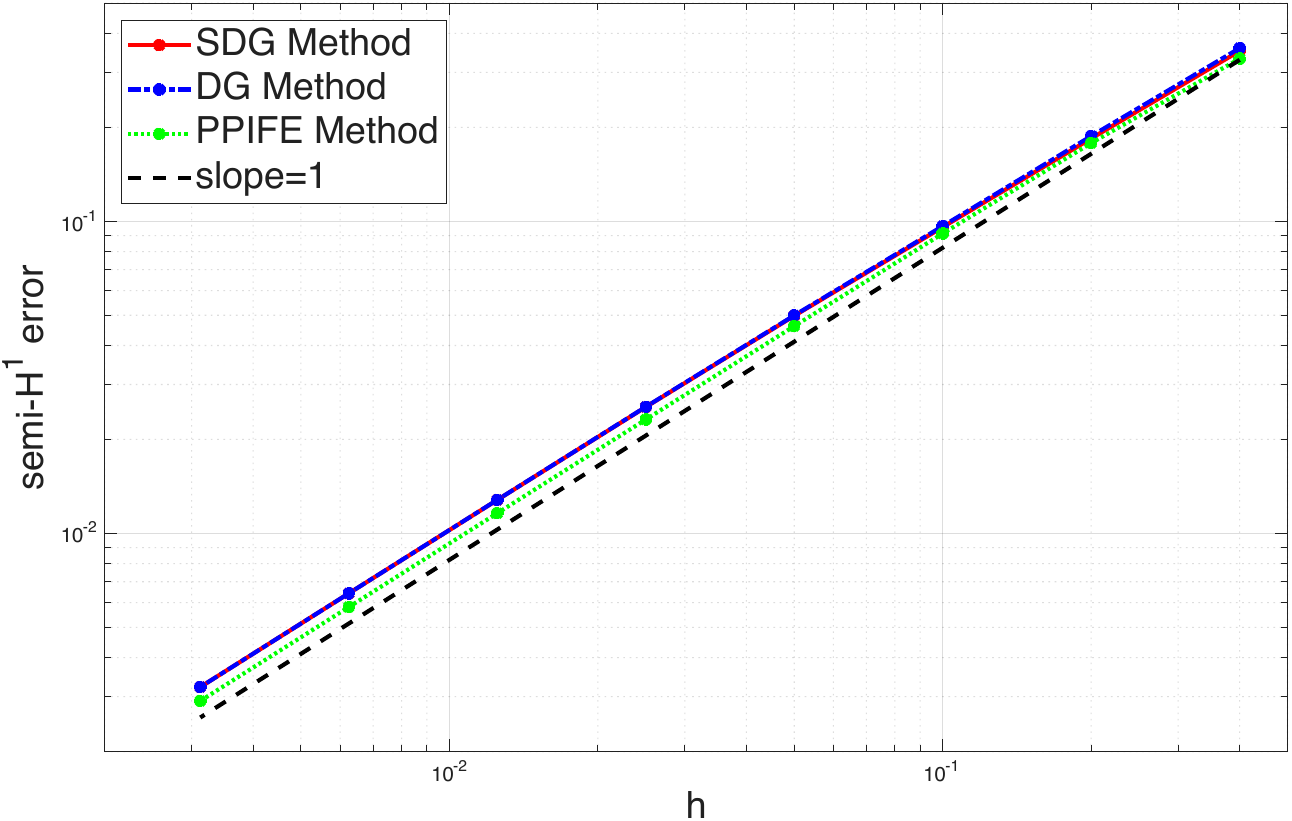}
        \caption{Semi-$H^1$ errors with $\beta=(1,10).$}
    \end{subfigure}

    \vspace{0.5em}

    \begin{subfigure}{0.48\textwidth}
        \centering
        \includegraphics[width=\textwidth]{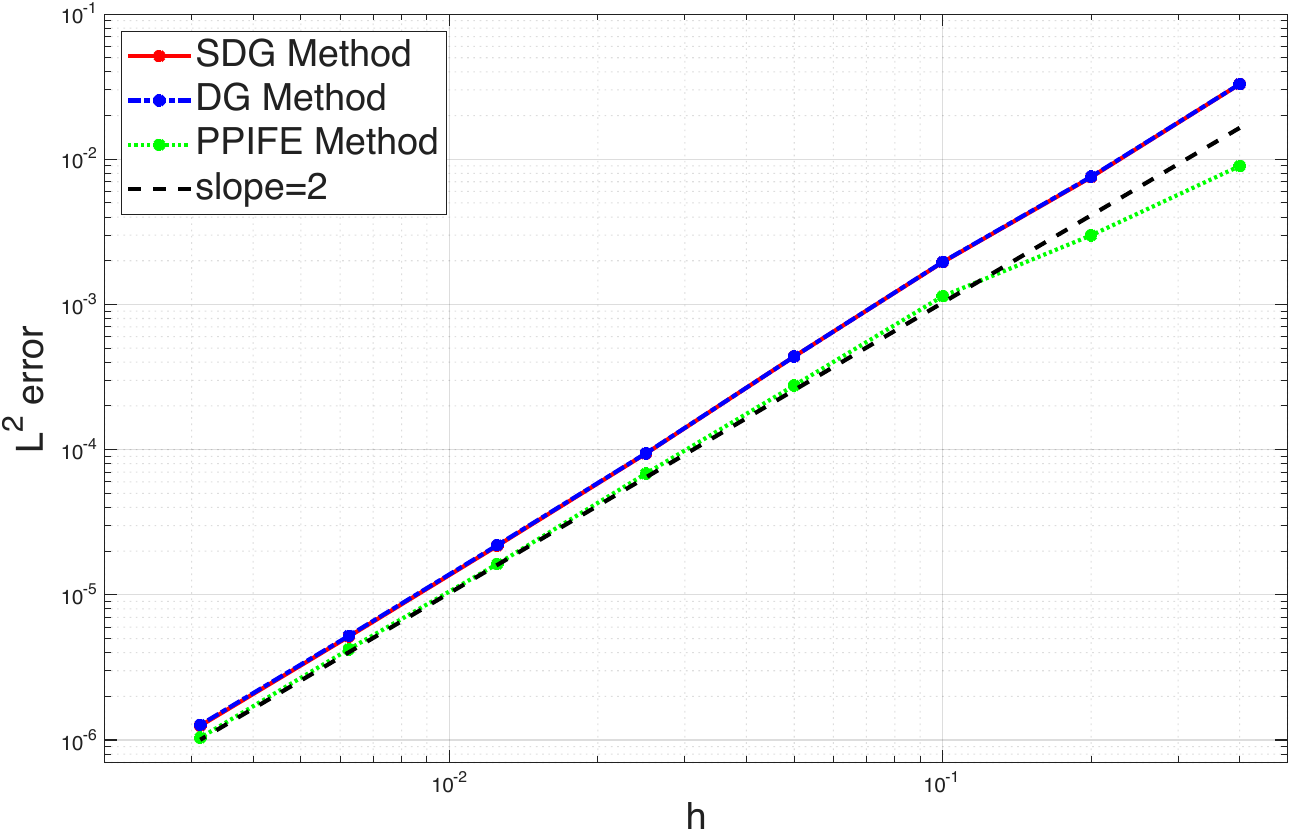}
        \caption{$L^2$ errors with $\beta=(1,1000).$}
    \end{subfigure}
    ~~
    \begin{subfigure}{0.48\textwidth}
        \centering
        \includegraphics[width=\textwidth]{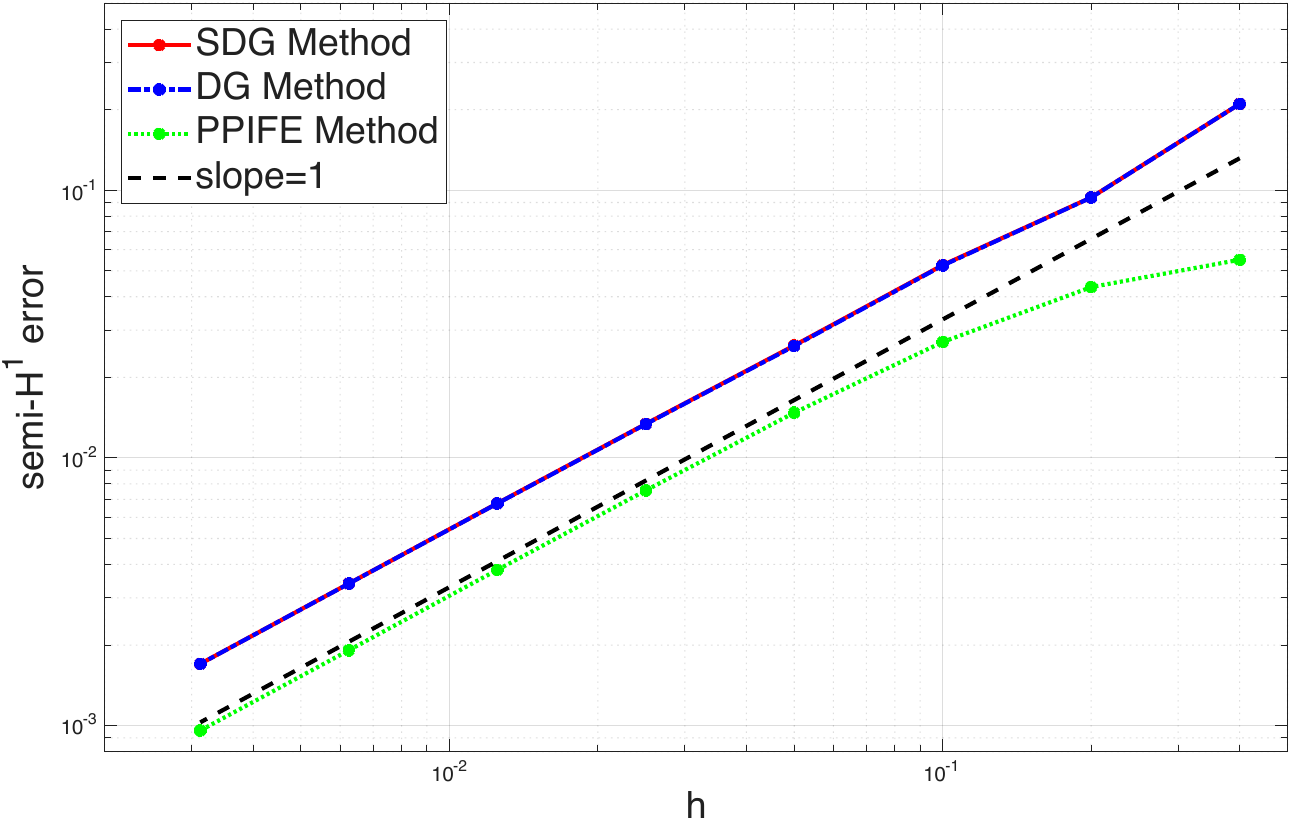}
        \caption{Semi-$H^1$ errors with $\beta=(1,1000).$}
    \end{subfigure}

    \caption{$L^2$ and semi-$H^1$ errors of the nonsymmetric SDG, DG, PPIFE methods with $\beta=(1,10)$  and $\beta=(1,1000)$ for $m=1$.}
    \label{fig8}
\end{figure} 

\begin{figure}[!htb]
    \centering

    \begin{subfigure}{0.48\textwidth}
        \centering
        \includegraphics[width=\textwidth]{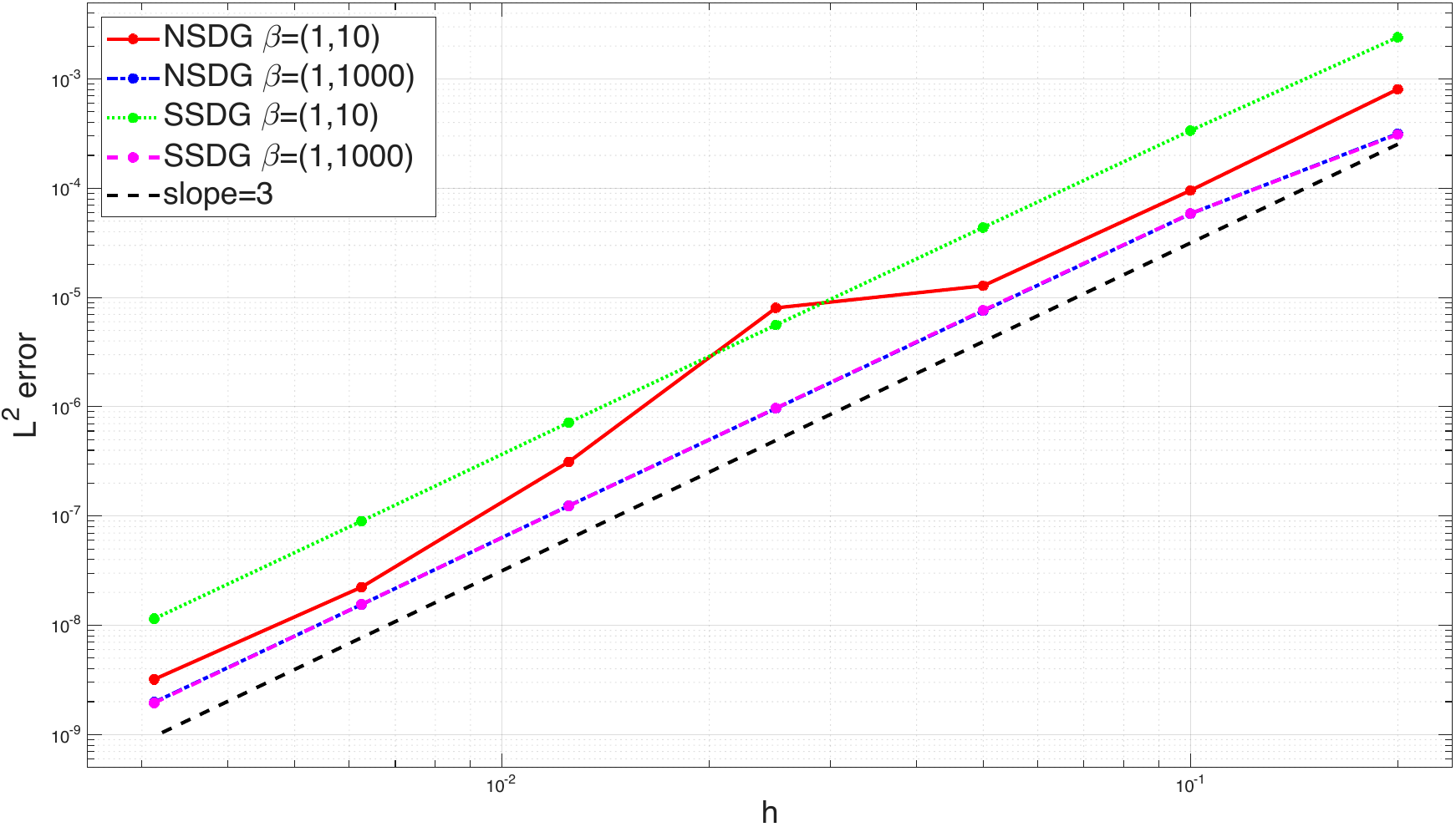}
        \caption{$L^2$ Errors }
    \end{subfigure}
    ~~
    \begin{subfigure}{0.48\textwidth}
        \centering
        \includegraphics[width=\textwidth]{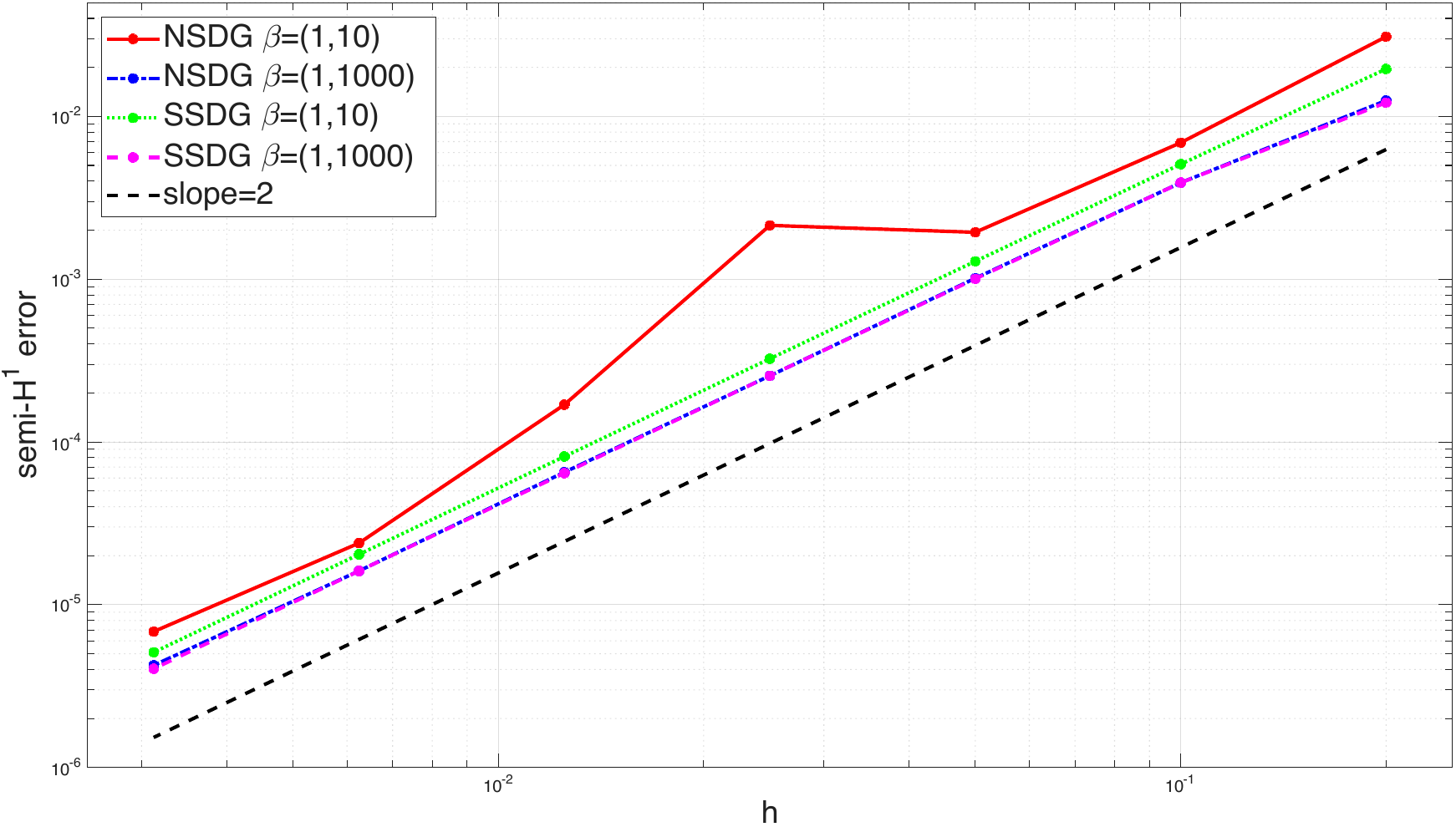}
        \caption{Semi-$H^1$ Errors}
    \end{subfigure}

    \caption{The $L^2$ and semi-$H^1$ errors of the nonsymmetric and symmetric SDG methods with $\beta=(1,10)$ and $\beta=(1,1000)$ for $m=2$}
    \label{fig9}
\end{figure}

\subsection{Star shaped interface}
\label{subsec:star_shape}
To assess the robustness of the SDG method under more challenging geometry, we replace the circular interface with a star-shaped one. 
Let $\Omega=(-2,2)^2$ and the level-set function defined in polar coordinates by
\begin{equation}
\psi(r,\theta)=
r^4\left(1+b \sin(6\theta)\right)^2-r_0,
\end{equation}
which defines a star-shaped interface 
$\Gamma=\{\psi=0\}$ with six lobes of variable curvature. We set
$(b,r_0)=\left(\frac{3}{10},\frac{\pi}{3}\right)$ and 
$\sigma_B^0=4\max\{\beta^-,\beta^+\}$. 
The  exact solution is 
\begin{equation}
\begin{aligned}
u(r,\theta)=
\begin{cases}
\frac{1}{\beta^-} \cos(\psi(r,\theta)) & \text{in}~ \Omega^-,\\
\frac{1}{\beta^+} \cos(\psi(r,\theta))
+  \left(\frac{1}{\beta^-}-\frac{1}{\beta^+}\right)
& \text{in}~ \Omega^+.\\
\end{cases}\label{5.2}
\end{aligned}
\end{equation}

Figure \ref{fig:star_shaped_interface} (a) shows the exact solution from \eqref{5.2}, and Figure \ref{fig:star_shaped_interface} (b) shows the S-SDG solution at $N=160$ with $\beta =(1,10)$. These two figures are visually indistinguishable. 
Figure \ref{fig:errors_star_shaped_interface} reports the convergence rates of both S-SDG and N-SDG methods for $m=1,2$. 

\begin{figure}[!htb]
    \centering

    \begin{subfigure}{0.48\textwidth}
        \centering
        \includegraphics[width=\textwidth]{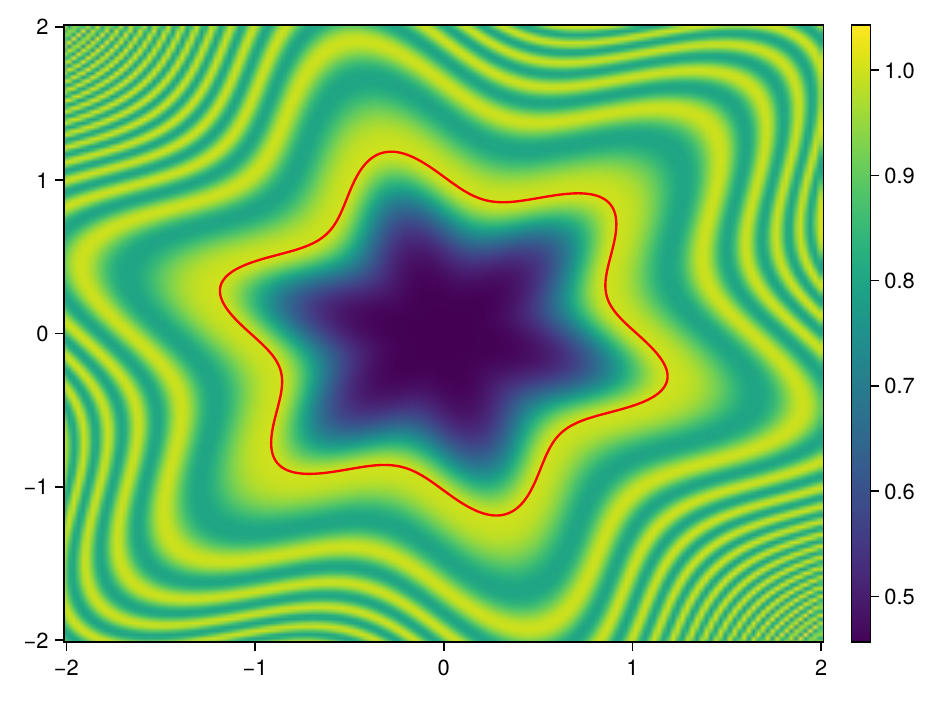}
        \caption{Exact Solution }
    \end{subfigure}
    ~~
    \begin{subfigure}{0.48\textwidth}
        \centering
        \includegraphics[width=\textwidth]{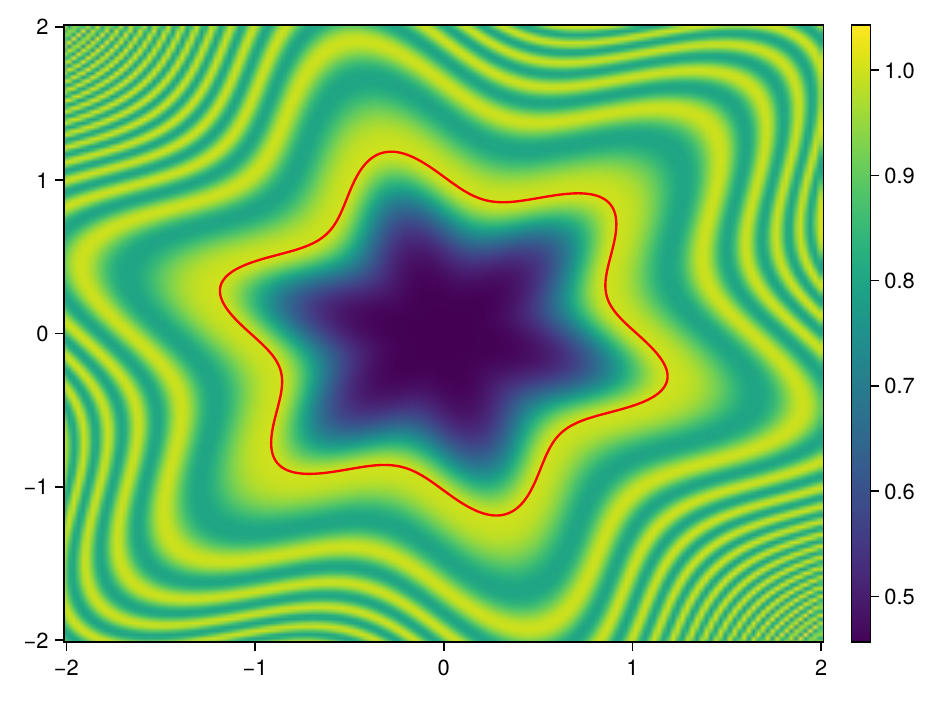}
        \caption{Numerical Solution}
    \end{subfigure}

    \caption{The exact solution and the numerical solution obtained using the SSDG method. The interface is shown in red.}
    \label{fig:star_shaped_interface}
\end{figure}

\begin{figure}[!htb]
    \centering

    \begin{subfigure}{0.48\textwidth}
        \centering
        \includegraphics[width=\textwidth]{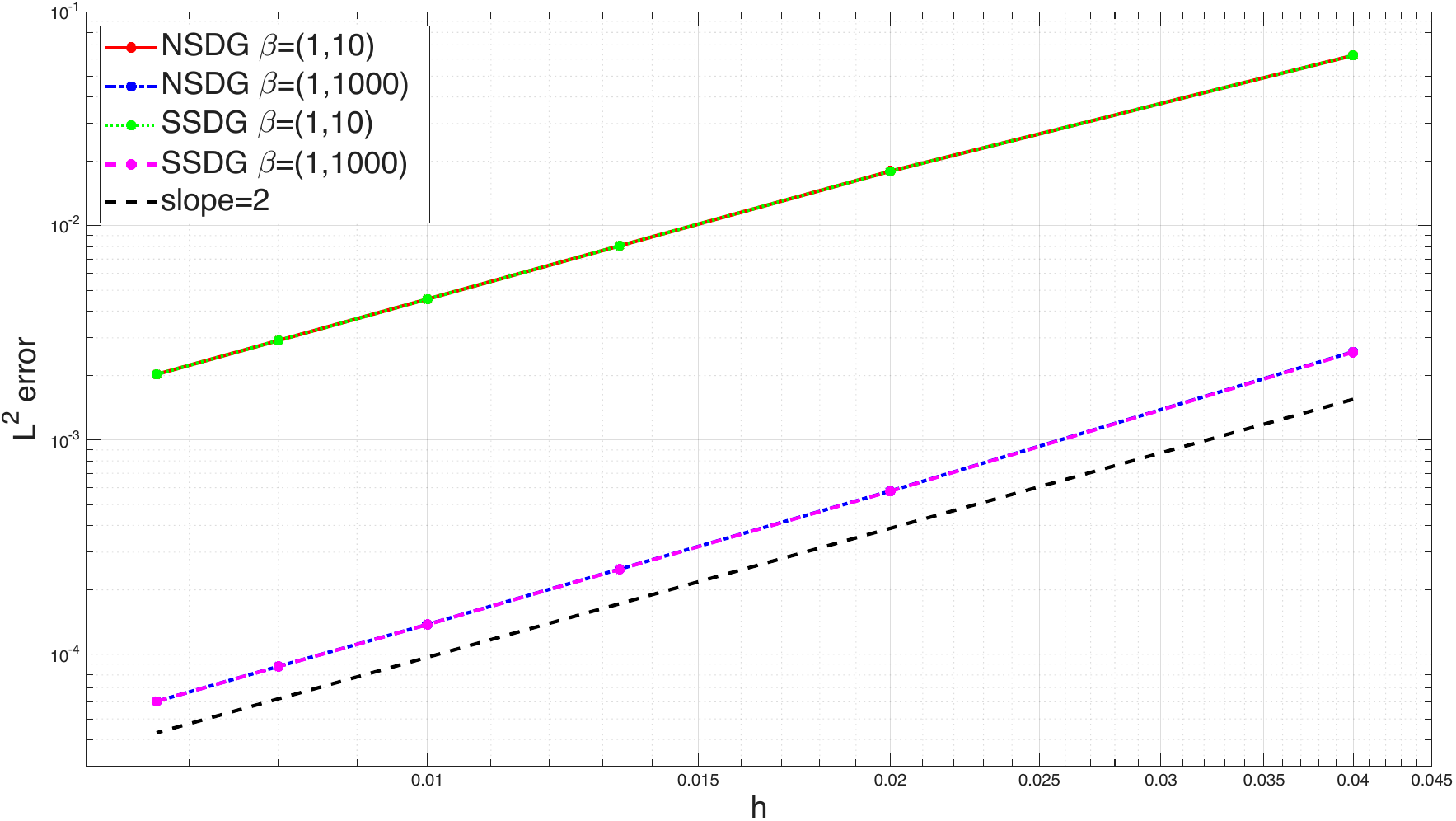}
        \caption{$m=1$ }
    \end{subfigure}
    ~~
    \begin{subfigure}{0.48\textwidth}
        \centering
        \includegraphics[width=\textwidth]{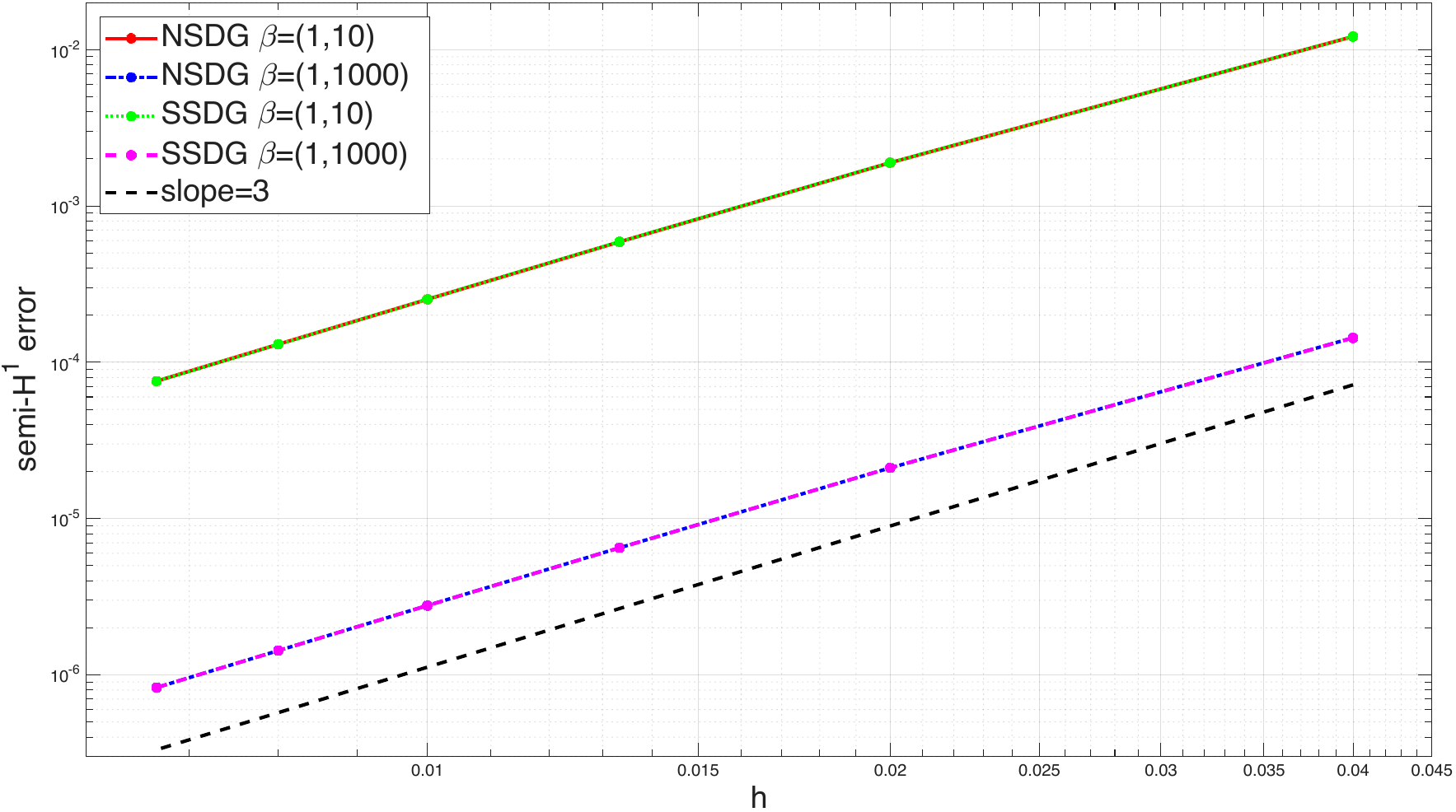}
        \caption{$m=2$}
    \end{subfigure}

    \caption{$L^2$ and semi-$H^1$ errors of the S-SDG and N-SDG methods for $\beta=(1,10)$ and $\beta=(1,1000)$, $m=1,2$.}
    \label{fig:errors_star_shaped_interface}
\end{figure}

\subsection{Eigenfunction solution on a circular interface}
\label{subsec:eigen_mode_shape}

In the previous two examples the solution depends only on the level-set function $\psi$. Consequently, on each interface element $K$, the transformed solution $\hat{u}(\eta,\xi)=(u\circ P_\Gamma)(\eta,\xi)$ depends only on the normal coordinate $\eta$ and not on the tangential coordinate $\xi$. We now consider a solution for which $\hat u$ depends on both $\eta$ and $\xi$ to demonstrate the ability of the SDG method to handle
more general solutions.

Let $\Omega=(-1,1)^2$ with the circular interface
$\Gamma$ defined in the first example.
The exact solution is given in polar coordinates by
\begin{equation}
u(r,\theta)=
\begin{cases}
J_1\left(\gamma^- r\right)\sin(\theta), & r\le r_0,\\
A J_1(\gamma^+ r) \sin(\theta)+B H_1(\gamma^+ r)\sin(\theta), & r>r_0.
\end{cases}
\label{eqn:bessel_solution}
\end{equation}
Here,  $\gamma^{\pm}=(\beta^{\pm})^{-\frac{1}{2}}$, $J_1$ is the Bessel function of the first kind of order one, and $H_1$ is the Hankel function of the second kind of order one.  
The coefficients $A$ and $B$ are determined by the interface jump
conditions through the linear system
\[
\begin{cases}
   AJ_1(\gamma^+r_0)+BH_1(\gamma^+r_0)
   =J_1(\gamma^-r_0),\\ 
   \gamma^+\left(AJ_1'(\gamma^+r_0)+BH_1'(\gamma^+r_0)\right)
   =\gamma^-J_1'(\gamma^-r_0).
\end{cases}
\]

A direct calculation shows that $f=u$ since $u$ is an
eigenfunction of $-\beta \Delta$ on each subdomain.
We solve the resulting interface problem using the S-SDG and N-SDG methods
with $\sigma_B^0=\max\{\beta^-,\beta^+\}$ on a sequence of uniformly refined meshes. 

Figure \ref{fig:errors_eigenmode_shaped_interface} confirms optimal convergence for both methods for $m=1$ and $m=2$. The tangential variation of $\hat u$ activates the cross terms in the Laplacian \eqref{2.2} involving $\hat u_{\xi\xi}$, $\hat u_\xi$, and the curvature-dependent coefficients $J_0, J_1, J_2$. Optimal rates are retained without modification of the scheme, which indicates that the Frenet IFE construction, the hybrid HIFE space built on it, and the SDG methods are robust for solutions with more general structure.

\begin{figure}[!htb]
    \centering

    \begin{subfigure}{0.48\textwidth}
        \centering
        \includegraphics[width=\textwidth]{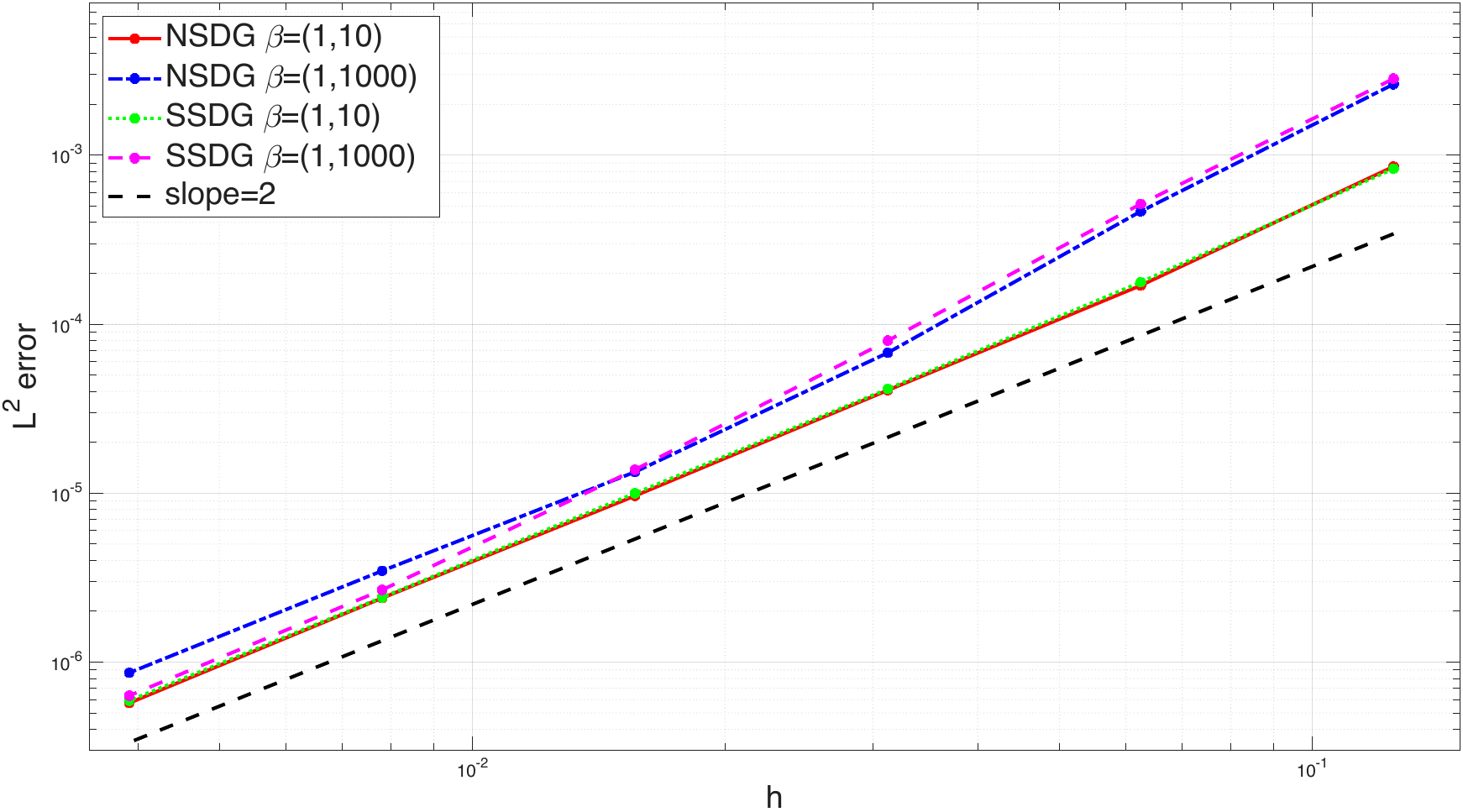}
        \caption{$m=1$ }
    \end{subfigure}
    ~~
    \begin{subfigure}{0.48\textwidth}
        \centering
        \includegraphics[width=\textwidth]{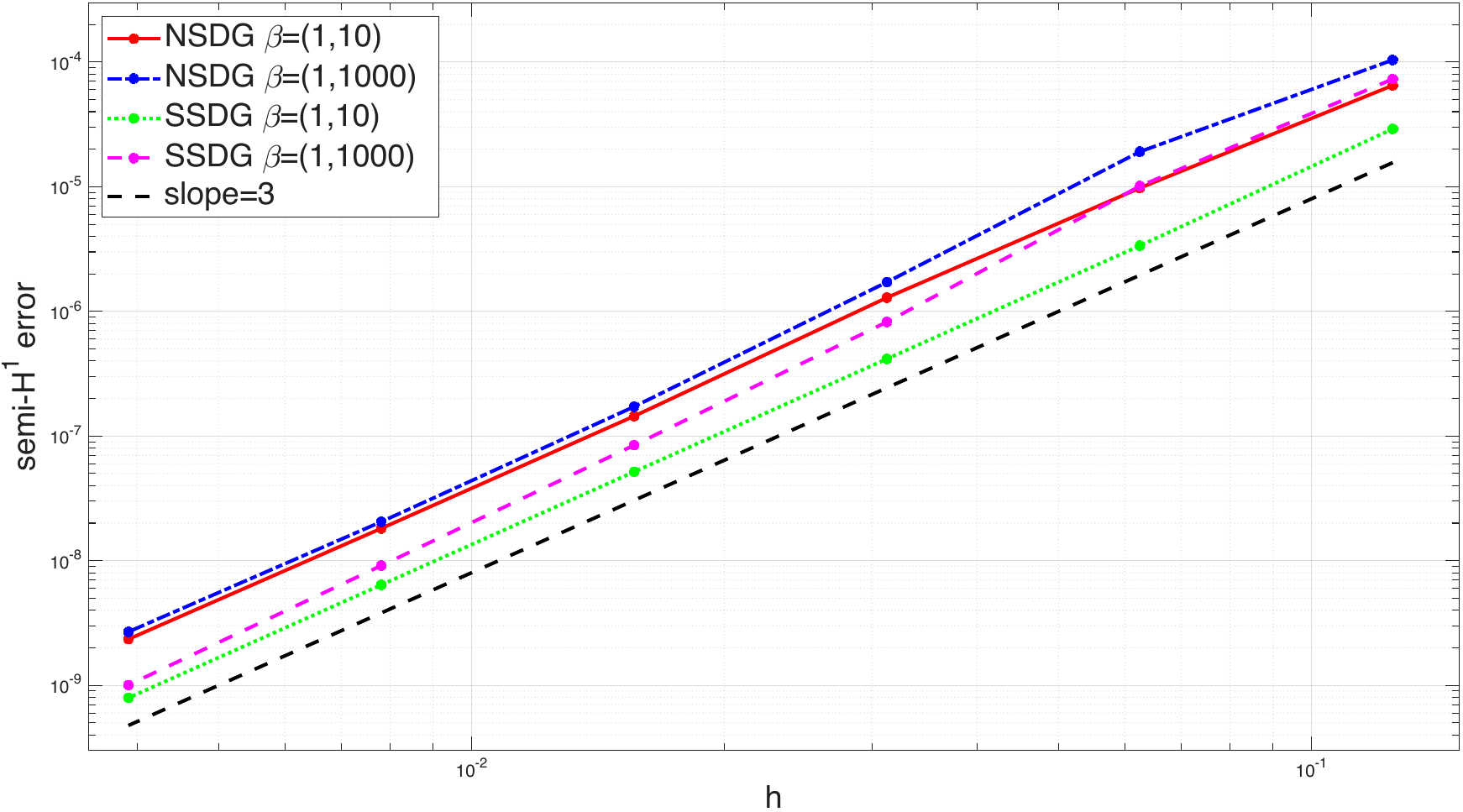}
        \caption{$m=2$}
    \end{subfigure}

    \caption{$L^2$ and semi-$H^1$ errors of the symmetric and nonsymmetric SDG methods with $\beta=(1,10)$ and $\beta=(1,1000)$ for $m=1,2$}
    \label{fig:errors_eigenmode_shaped_interface}
\end{figure}
\section{Conclusion}\label{sec6}
In this work, we present a high-order selective discontinuous Galerkin method that applies the DG formulation only on interface elements. Compared with the standard discontinuous Galerkin method, it significantly reduces the computational cost. We construct a new high-order hybrid IFE space that is locally $H^1$-conforming and has optimal approximation properties. We also establish rigorous error estimates in both the energy norm and the $L^2$ norm. Finally, the numerical results are consistent with the theoretical analysis.

\section*{Acknowledgements}
Xu Zhang is partially supported by National Science Foundation grant DMS-2110833.

\bibliography{FangLiuBibDesk}

\end{document}